\documentclass{article}  
\usepackage[pdftex,
bookmarksnumbered,
bookmarksopen,
colorlinks,
citecolor=blue,
linkcolor=blue,]{hyperref}
\usepackage{amsmath}
\usepackage{amssymb}
\usepackage{mathrsfs}
\usepackage{amsfonts}
\usepackage{amsthm}
\usepackage{algorithm}
\usepackage{booktabs}
\usepackage{listings}
\usepackage{boxedminipage}
\usepackage{algorithmic}
\usepackage{stmaryrd}
\usepackage{cite}
\usepackage{relsize}
\usepackage{lscape}
\usepackage[top=1.3in, bottom=1.5in, left=1.5in, right=1.5in]{geometry}
\usepackage{stmaryrd} 
\usepackage{relsize}%
\usepackage{url}
\usepackage{color,xcolor}
\usepackage{epsfig}
\usepackage{graphicx}
\usepackage{subfigure}
\usepackage{epstopdf}
\usepackage{braket}
\usepackage{tikz}
\usetikzlibrary{shapes,arrows}
\tikzstyle{nomalnode} = [rectangle,rounded corners, minimum width=3cm,minimum height=1cm,text centered, draw=black]
\tikzstyle{arrow} = [thick,->,>=stealth]

\usepackage{cleveref}

\newtheorem{theorem}{Theorem}[section]
\newtheorem{definition}{Definition}[section]
\newtheorem{proposition}{Proposition}[section]
\newtheorem{property}{Property}[section]

\newtheorem{remark}{Remark}[section]

\newcommand{\R}{\mathbb{R}}
\newcommand{\ot}{\underset{\bf t}{\otimes}}
\newcommand{\oh}{\underset{\bf H}{\otimes}}

\usepackage{amsopn}
\DeclareMathOperator{\diag}{diag}

\DeclareMathOperator*{\argmin}{arg\,min}                                  
                                                              
	\definecolor{darkgray}{rgb}{0.66, 0.66, 0.66}

\title{Hot-SVD: Higher-Order t-Singular Value Decomposition for Tensors based on Tensor-Tensor Product}

\author{Ying Wang$^*$, Yuning Yang\thanks{College of Mathematics and Information Science, Guangxi University, Nanning, 530004, China} \thanks{Corresponding author: Yuning Yang, yyang@gxu.edu.cn}                             }

\begin{document} 
\maketitle


\begin{abstract}
  This paper considers a way of generalizing   the t-SVD of third-order tensors (regarded as tubal matrices) to tensors of arbitrary order
  \(N\) (which can be similarly regarded as tubal tensors of order \(N-1\)). \color{black}Such a generalization is different from the   t-SVD for  tensors of order greater than three [Martin, Shafer, Larue, SIAM J. Sci. Comput., 35 (2013), A474--A490]. \color{black} 
   The decomposition is 
  called Hot-SVD since it can   be recognized as a tensor-tensor product version of  HOSVD.
The existence of Hot-SVD is  proved. To this end,      a new transpose for third-order tensors
  is introduced. This transpose is crucial in the verification of Hot-SVD, since it serves as a bridge between tubal tensors and their unfoldings.
  We establish some   properties of Hot-SVD, analogous to those of HOSVD, and in doing so we emphasize the perspective of tubal tensors.
  The truncated and sequentially truncated  Hot-SVD are then introduced, whose    error bounds are $\sqrt{N}$ for an $(N+1)$-th order tensor. 
  We provide numerical examples to validate Hot-SVD, truncated Hot-SVD, and sequentially truncated Hot-SVD.
  
  \end{abstract}

 \hspace{2mm}\vspace{3mm}

\section{Introduction}
Research in tensor decomposition and approximation has seen increasing popularity with the
exponential increase and availability of data in our world.  
The prospect of 
encoding these data in a tensor-based format allows us to exploit fully their inherent 
multi-dimensional features.
Tensor decompositions, such as canonical polyadic decomposition, Tucker decomposition, tensor-train decomposition, t-SVD, have found various applications in  signal processing, machine learning, computer vision, etc; see, e.g, \cite{kolda2010tensor,comon2014tensors,cichocki2015tensor,sidiropoulos2017tensor,oseledets2011tensor,kilmer2011factorization} and the references therein.  
One of the
powerful tensor-algebraic methods for processing tensor-type data is  
provided by the t-product based tensor theory and computation \cite{braman2010third,kilmer2011factorization,kilmer2013third,kernfeld2015tensor}, for which an Eckart-Young type result (t-SVD) exists.
The t-product framework has recently generated a surge of research activities both in theory and applications, rendering
available a large stock of the matrix-algebra arsenal to the tensor-algebra community
(see, e.g,, \cite{braman2010third,kilmer2011factorization,kilmer2013third,kernfeld2015tensor,miao2020generalized,lund2020tensor,miao2021t,zheng2021t,qi2021tubal,ling2021st,qi2021t,zhu2022tensor,newman2018stable,zhang2017exact,zhang2014novel,hao2018tschatten,ming2019multiview}).

As \cite{kernfeld2015tensor,qi2021tubal}, a natural perspective on the t-product (and also the more general tensor-tensor
product defined by an arbitrarily invertible linear transformation) is provided by tubal matrices. That is, one regards a third-order tensor
\(\mathcal{A}\), say of size \(I_1 \times I_2 \times I_3\), as a tubal matrix of size \(I_1 \times I_2\) whose entries are horizontal vectors (tubal scalars) of  length $I_3$. 
A variation of the Hadamard product, which is 
obtained by composing the Hadamard product with an invertible linear transformation, gives the general tensor-tensor product
of two tubal scalars  \cite{kernfeld2015tensor}. When the linear transformation in action is the (non-normalized) discrete Fourier transform (DFT),
this gives the most widely used t-product of tubal scalars, which can then be extended to tubal matrix multiplications to give the t-product of tubal matrices. 
Based on the t-product, the t-SVD of a third-order tensor was proposed in \cite{kilmer2011factorization}. The idea is to decompose the
frontal slices of a tubal matrix in the Fourie domain using the matrix SVD, in conjunction with DFT and IDFT (inverse 
discrete Fourier transform) operations on the tubal scalars before and after the matrix decompositions. 
t-SVD was then generalized to tensors of order higher than three via recursion \cite{martin2013order}. t-SVD was introduced in \cite{kilmer2021tensor} in the sense of the more general tensor-tensor product.

In this paper, we investigate an alternative   way of generalizing t-SVD to tensors of order higher than three. It is known that the celebrated higher-order singular value decomposition (HOSVD) is quite a successful extension of SVD to higher-order tensors that decomposes the data tensor into a core tensor and a collection of factor matrices. HOSVD together with the tensor-tensor product naturally suggests us to consider a tensor-tensor product version of HOSVD. The new decomposition, called the higher-order t-singular value decomposition (Hot-SVD), treats an $N$-th order tensor as an $(N-1)$-th order tubal tensor, and factorizes it into a core tubal tensor and a collection of tubal matrices. This is essentially based on the unfolding in the tensor-tensor product sense and t-SVD for third-order tensors (tubal matrices).  The connnection of Hot-SVD with t-SVD and HOSVD is illustrated in Fig. \ref{fig:connection_hotsvd_tsvd_hosvd}.

\begin{figure}
\centering
\begin{tikzpicture}[node distance=4.5cm,scale=0.75, transform shape]
  \node (node1) [nomalnode] {SVD of matrices};
  \node (node2) [nomalnode,right of=node1,xshift=3cm] {HOSVD of tensors};
  \node (node3) [nomalnode,below of=node1] {t-SVD of tubal matrices};
  \node (node4) [nomalnode,below of=node2] {Hot-SVD of tubal tensors};
  
  \draw [arrow] (node1) -- node[anchor=south] {higher-order} (node2);
  \draw [arrow] (node1) -- node[anchor=east] {tubal} (node3);
  \draw [arrow] (node2) -- node[anchor=west] {tubal} (node4);
  \draw [arrow] (node3) -- node[anchor=north] {higher-order} (node4);
\end{tikzpicture}
\caption{Connection among SVD, t-SVD, HOSVD, and Hot-SVD.}\label{fig:connection_hotsvd_tsvd_hosvd}
\end{figure}
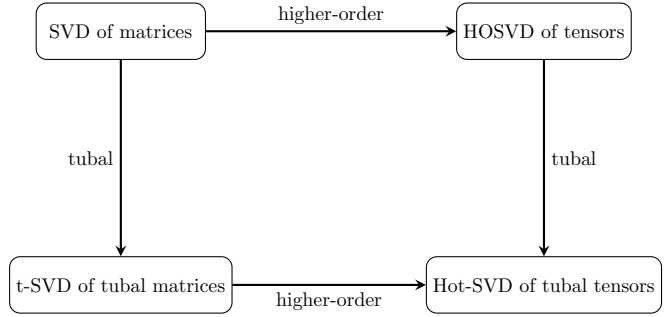

We then prove the existence of Hot-SVD and establish several   properties such as all-orthogonality and ordering that are similar to those of HOSVD. A crucial point to make the analysis go through is the use of a   transpose for third-order tensors that is different from  \cite[Def. 3.14]{kilmer2011factorization} and \cite[Def. 4.4]{kernfeld2015tensor}.
We show that some necessary properties   hold for this   transpose. In particular, based on this transpose, the Kronecker product and corresponding properties can be generalized in the sense of tensor-tensor product. These enable us to build a bridge between tubal tensors and their tubal matrix unfoldings such that the proof of validness of Hot-SVD can go through.

We then introduce the truncated and sequentially truncated Hot-SVD, and show that their error bounds are $\sqrt{N}$ for an $(N+1)$-th order tensor. The sequentially truncated Hot-SVD generalizes the sequentially truncated HOSVD to tubal tensors \cite{vannieuwenhoven2012new}.

The rest is organized as follows. In \cref{sec:pre}, we collect some preliminary materials related to the tensor-tensor product of 
third-order tensors.
In Section \ref{sec:smallt}, we introduce the new transpose   and prove its basic properties. 
 We also extend some familiar notions and
results from matrices (such as the Kronecker product) to tubal matrices. 
The Hot-SVD model for tubal tensors is established in Section \ref{sec:main}, utilizing the new transpose for third-order tensors. 
 Section \ref{sec:prop} is mainly focused on the truncated Hot-SVD,  the sequentially truncated Hot-SVD and their error bounds. We analyze the computational complexity of Hot-SVD, truncated Hot-SVD, and sequentially truncated Hot-SVD in Section \ref{sec:comp}. Numerical examples are provided in Section \ref{sec:numer}. 
Finally we draw some conclusions in Section \ref{sec:con}.

\section{Preliminaries}
\label{sec:pre}

In this section, preliminary materials related to the tensor-tensor product of 
third-order tensors are recorded.

  A tensor of order \(N\) is an element \(\mathcal{A} \in \R^{I_1 \times I_2 \times \cdots \times I_N}\), i.e.,
  a multi-way array \(\mathcal{A}=(a_{i_1i_2\cdots i_N})_{(i_1,i_2,\ldots,i_N)\in I_1 \times I_2 \times \cdots \times I_N}\).
  Complex tensors are defined similarly.
The Matlab indexing notation is adopted in this article. Denote by \(\mathcal{A}(i_1,\ldots,i_N)\) the \((i_1,\ldots,i_N)\)-th entry of 
\(\mathcal{A}\) and by \(\mathcal{A}(i_1,\ldots,i_{n-1},:,i_{n+1},\ldots,i_N)\) the mode-\(n\) fiber of \(\mathcal{A}\) obtained by
varying the \(n\)-th index while
fixing the other \(N-1\) indices to be \(i_1,\ldots,i_{n-1}, i_{n+1},\ldots,i_{N}\) respectively.

  The \(n\)-mode product of a tensor \(\mathcal{A} \in \R^{I_1 \times I_2 \times \cdots \times I_N}\)
  with a matrix \(U \in \R^{J \times I_n}\) is the tensor 
  \(\mathcal{A}\times_n U \in \R^{I_1 \times \cdots \times I_{n-1}\times J \times I_{n+1} \times \cdots \times I_N}\)
  obtained by multiplying each mode-\(n\) fiber of \(\mathcal{A}\) by \(U\). 

\begin{definition}[{\cite[Definition 2.1]{kilmer2013third}}]
  An element \(\mathbf{c} \in \R^{1 \times 1 \times p}\) is called a tubal scalar of length \(p\). That is, a tubal scalar of length \(p\) is 
  just a vector of dimension \(p\) viewed as a third-order tensor with a single tubal fiber. 
  Denote by \(\mathbf{c}^{(j)}\) the \(j\)-th frontal slice of \(\mathbf{c}\) (as a third-order tensor), i.e., the \(j\)-th component 
  of \(\mathbf{c}\) as a 
  vector. The set of tubal scalars of length \(p\) is denoted by \(\mathbb{R}_p\).
  Complex tubal scalars are defined similarly and the set of complex tubal scalars of length \(p\) is denoted by \(\mathbb{C}_p\). 
\end{definition}

For \(\mathbf{a}=(\mathbf{a}^{(1)},\ldots,\mathbf{a}^{(p)})~{\rm and}~ \mathbf{b}=(\mathbf{b}^{(1)},\ldots,\mathbf{b}^{(p)}) \in \mathbb{C}_p\),
their Hadamard product is the tubal scalar
\(\mathbf{a} \odot \mathbf{b} :=(\mathbf{a}^{(1)}\mathbf{b}^{(1)},\ldots,\mathbf{a}^{(p)}\mathbf{b}^{(p)}) \in \mathbb{C}_p\),  where  \(\odot\) refers to the Hadamard product. 
It can be easily verified that the Hadamard product is commutative, 
associative, unital (with the identity being \( (1,1,...,1) \in \mathbb{C}_p\)), and distributive over addition. 

The following definition comes from \cite[Definition 4.2]{kernfeld2015tensor}, specialized to the tubal scalar case.
\begin{definition}[\cite{kernfeld2015tensor}]\label{def:tproductubalscalar}
  Let \(L:\mathbb{C}_p \to \mathbb{C}_p \) be an invertible linear transformation. For tubal scalars 
  \(\mathbf{a}, \mathbf{b} \in \mathbb{C}_p\), 
  their tensor-tensor product with respect to \(L\) is the tubal scalar 
  \(\mathbf{a} *_L \mathbf{b} = L^{-1}(L(\mathbf{a}) \odot L(\mathbf{b})) \in \mathbb{C}_p\).
\end{definition}

When the invertible linear transformation
\(L\) is the (non-normalized) discrete Fourier transform (DFT), the resulting product is the popular t-product defined in \cite{kilmer2011factorization}.
There are other choices of \(L\) such as discrete cosine transform and discrete wavelet transform \cite{kernfeld2015tensor}. 
For ease of notation, in the sequel, we will drop the subscript \(L\) and denote
by \(\mathbf{a}*\mathbf{b}\) the tensor-tensor product of \(\mathbf{a}\) and \(\mathbf{b}\) with respect to a fixed invertible linear transformation \(L\).

The tensor-tensor product of tubal scalars is commutative, associative, unital (with the identity being 
\(L^{-1}\left((1,1,\ldots,1)\right) \in \mathbb{C}_p\)), and distributive over addition \cite{kernfeld2015tensor}. 
Therefore, \(\mathbb{C}_p\) endowed with the tensor-tensor product is a commutative ring (\cite[Proposition 4.2]{kernfeld2015tensor}), but not a field in general.

\begin{definition}[{\cite[Definition 2.4]{kilmer2013third}}]
  A tubal matrix with entries in \(\mathbb{C}_p\)\ is a two-dimensional array \(\mathcal{A} = (\mathbf a_{ij})_{(i,j) \in I \times J}
  \in \mathbb{C}_p^{I \times J}\) where the entries \(\mathbf a_{ij} \in \mathbb{C}_p\) are tubal scalars of length \(p\).
  Thus a tubal matrix \(\mathcal{A} \in \mathbb{C}_p^{I \times J}\) is essentially a third-order tensor
  \(\mathcal{A} \in \mathbb{C}^{I \times J \times p}\).
\end{definition}

\begin{definition}[{\cite[Definition 2.1]{kernfeld2015tensor}}]
  Let \(\mathcal{A} \in \mathbb{C}_p^{I \times J}\) and \(\mathcal{B} \in \mathbb{C}_p^{J \times K}\) be two tubal matrices.
  The face-wise product of \(\mathcal{A}\) and \(\mathcal{B}\) is the tubal matrix 
  \(\mathcal{A} \Delta \mathcal{B} \in \mathbb{C}_p^{I \times K}\) defined according to 
  \[\left(\mathcal{A} \Delta \mathcal{B} \right)^{(i)} = \mathcal{A}^{(i)}\mathcal{B}^{(i)},\,\,\,i=1,\ldots,p,\]
  where \(\mathcal{A}^{(i)}\) is the \(i\)-th frontal slice of \(\mathcal{A}\).
\end{definition}

\begin{definition}[{\cite[Definition 4.2]{kernfeld2015tensor}}]\label{def:ttproduct}
  Let \(\mathcal{A} \in \mathbb{C}_p^{I \times J}\) and \(\mathcal{B} \in \mathbb{C}_p^{J \times K}\) be two tubal matrices.  
  The 
  tensor-tensor product of \(\mathcal{A}\) and \(\mathcal{B}\)
  is the tubal matrix \(\mathcal{A}*\mathcal{B} \in \mathbb{C}_p^{I \times K}\) given by 
  \[\mathcal{A} * \mathcal{B} = L^{-1}(L(\mathcal{A}) \Delta L(\mathcal{B})),\]
  where \(L(\mathcal{A})\) means applying \(L\) to each entry of \(\mathcal{A}\), and \(\Delta\) is the 
  face-wise product.
\end{definition}

The tensor-tensor product of tubal matrices can also be defined using the standard matrix multiplication rule.
\cite[Lemma 4.1]{kernfeld2015tensor} states that for any \(i \in I, k \in K\),
\begin{equation}\label{eq:matrixproduct}
  \left(\mathcal{A}*\mathcal{B}\right)(i,k) = \sum_{j=1}^{J}\mathcal{A}(i,j)*\mathcal{B}(j,k).
\end{equation}

When \(L\) is the DFT, \cref{def:ttproduct} is equivalent to the t-product defined in \cite{kilmer2011factorization}.

The tensor-tensor product of tubal matrices, like the usual matrix multiplication, is associative and distributive over addition
(see the proof of \cite[Proposition 4.2]{kernfeld2015tensor}), but not commutative 
in general. 

\begin{definition}[{\cite[Proposition 4.1]{kernfeld2015tensor}}]
  Let \(\widehat{\mathcal{I}} \in \mathbb{C}_p^{I \times I}\)be the tubal matrix  such that  
  \(\widehat{\mathcal{I}}^{(i)}\)
  is the identity matrix for \(i=1,\ldots,p\). The identity tubal matrix with respect to the tensor-tensor product
  is defined to be \(\mathcal{I} = L^{-1}(\widehat{\mathcal{I}})\).
  It holds that \(\mathcal{A}*\mathcal{I}=\mathcal{A}\) for any \(\mathcal{A} \in \mathbb{C}_p^{H \times I}\)
  and \(\mathcal{I}*\mathcal{B}=\mathcal{B}\) for any \(\mathcal{B} \in \mathbb{C}_p^{I \times J}\).
\end{definition}

When \(L\) is the DFT, the identity tubal matrix \(\mathcal{I}\) is the third-order tensor whose first frontal slice is the identity matrix and 
whose other frontal slices are zero matrices.

\begin{definition}[{\cite[Definition 4.4]{kernfeld2015tensor}}]\label{def:herm}
  The Hermitian transpose \(\mathcal{A}^H \in \mathbb{C}_p^{J \times I} \) of 
  a tubal matrix \(\mathcal{A} \in \mathbb{C}_p^{I \times J}\) is defined according to  
  \[\left(L(\mathcal{A}^H)\right)^{(i)} = \left(L(\mathcal{A})^{(i)}\right)^H,\,\,\,i=1,\ldots,p,\]
  where \(\left(L(\mathcal{A})^{(i)}\right)^H\) is the Hermitian transpose of the matrix \(L(\mathcal{A})^{(i)}\).  
 In particular, the Hermitian transpose of a tubal scalar \(\mathbf{a} \in \mathbb{C}_p=\mathbb{C}_{p}^{1 \times 1}\) is the tubal scalar \(\mathbf{a}^H\) such that \(\left(L(\mathbf{a}^H)\right)^{(i)}=
 \left(L(\mathbf{a})^{(i)}\right)^H
 =\overline{L(\mathbf{a})^{(i)}}\) for \(i=1,\ldots,p.\)
\end{definition}

When \(\mathcal{A}\) is real-valued, we write \(\mathcal{A}^{T}\) for \(\mathcal{A}^H\) and 
\(\mathcal{A}^{T}\) is called the transpose of \(\mathcal{A}\).
When \(L\) is the DFT, the transpose \(\mathcal{A}^{T}\) can be obtained by transposing each frontal slice of \(\mathcal{A}\) 
and reversing the order of the frontal slices except for the first one (see \cite[p. 560]{kernfeld2015tensor}).

\begin{proposition}[{\cite[Proposition 4.3]{kernfeld2015tensor}}]
  The Hermitian transpose enjoys the rule \((\mathcal{A}*\mathcal{B})^H=\mathcal{B}^H*\mathcal{A}^H.\)
\end{proposition}

\begin{definition}[{\cite[Definition 2.3]{kilmer2021tensor}}]\label{def:unitary}
  A tubal matrix \(\mathcal{A} \in \mathbb{C}_p^{I \times I}\) is said to be unitary with respect to the tensor-tensor product if 
  \(\mathcal{A}*\mathcal{A}^{H}=\mathcal{A}^{H}*\mathcal{A}=\mathcal{I}\), where \(\mathcal{I}\) is the identity tubal matrix 
  with respect to the tensor-tensor product. Orthogonality with respect to the tensor-tensor product is defined similarly using
  \(\mathcal{A}^{T}\) when \(\mathcal{A}\) is real-valued.
\end{definition}

\begin{definition}
  A tubal matrix \(\mathcal{A} \in \mathbb{C}_p^{I \times J}\) (\(I \ge J\)) is said to be partially unitary with respect to 
  the tensor-tensor product if 
  \(\mathcal{A}^H*\mathcal{A}=\mathcal{I}\). Similarly a real tubal tensor \(\mathcal{A}\in \mathbb{R}_p^{I \times J}\)
  (\(I \ge J\)) is said to be partially orthogonal with respect to 
  the tensor-tensor product if \(\mathcal{A}^{T}*\mathcal{A}=\mathcal{I}\).
\end{definition}

We quote here the singular value decomposition theorem for tubal matrices based on the tensor-tensor product.

\begin{theorem}[t-SVD of tubal matrices, {\cite[Theorem 5.1]{kernfeld2015tensor}}]
  Let \(\mathcal{A} \in \mathbb{C}_p^{I \times J}\) be a tubal matrix. Then there exist unitary tubal matrices \(\mathcal{U} \in \mathbb{C}_p^{I \times I}\) and 
  \(\mathcal{V} \in \mathbb{C}_p^{J \times J}\) and an f-diagonal tubal matrix \(\mathcal{S} \in \mathbb{C}_p^{I \times J}\) such that 
  \[\mathcal{A} = \mathcal{U}*\mathcal{S}*\mathcal{V}^H.\]
\end{theorem}
Here an f-diagonal tubal matrix means a tubal matrix whose frontal slices are diagonal matrices.

\begin{definition}[{\cite[Definition 3.4]{kilmer2021tensor}}]
  Let \(\mathcal{A} \in \mathbb{C}_p^{I \times J}\) be a tubal matrix and \(\mathcal{A} = \mathcal{U}*\mathcal{S}*\mathcal{V}^H\) 
  a t-SVD of \(\mathcal{A}\). The number of non-zero tubal scalars on the diagonal
  of \(\mathcal{S}\), which does not depend on the particular decomposition, is called the t-rank of \(\mathcal{A}\).

\end{definition}

\begin{definition}[{\cite[Definition 3.5]{kilmer2021tensor}}]
  Let \(\mathcal{A} \in \mathbb{C}_p^{I \times J}\) be a tubal matrix. The vector \(\rho = (\rho_1,\ldots,\rho_p)\),
  where \(\rho_i\) is the rank of the \(i\)-th frotal slice \(L(\mathcal{A})^{(i)}\), is called the multi-rank of \(\mathcal{A}\).
\end{definition}

An especially satisfyting feature of the t-SVD is the following Eckart-Young-type result.
The tubal matrix notation is used in the statement below.

\begin{theorem}[{\cite[Theorem 3.7]{kilmer2021tensor}}]\label{thm:EY}
  Let \(L\) be of the form \(L = cW\), where \(c \in \mathbb{C}\) is a non-zero scalar and \(W: \mathbb{C}_p \to \mathbb{C}_p \) is 
  a unitary transformation (i.e. a unitary matrix).\footnote{The DFT is of this form, with \(c=\sqrt{p}\) and \(W\) being 
  the normalized DFT matrix.}
  Let the t-SVD of \(\mathcal{A} \in \mathbb{C}_{p}^{n_1 \times n_2}\) be given by 
  \(\mathcal{A} = \mathcal{U}*\mathcal{S}*\mathcal{V}^H\) and for \(k < \min(n_1,n_2)\)
  define 
  \[\mathcal{A}_k = \sum_{i=1}^{k}\mathcal{U}(:,i)*\mathcal{S}(i,i)*\mathcal{V}(:,i)^H.\]
  Then \(\mathcal{A}_k = \argmin_{\widetilde{\mathcal{A}}\in M}\|\mathcal{A}-\widetilde{A}\|_F\), where
  \(M = \{ \mathcal{X}*\mathcal{Y} \, | \, \mathcal{X} \in \mathbb{C}_{p}^{n_1 \times k}, \mathcal{Y} \in \mathbb{C}_{p}^{k \times n_2} \}\).
\end{theorem}

The tubal matrix perspective naturally leads to the interpretation of a third-order tensor (i.e. a tubal matrix) as a linear operator
on the space of (lateral) matrices (i.e. tubal (column) vectors); see \cite{kilmer2013third} for a detailed exposition.

The following result is taken from \cite{kilmer2013third}, where it was stated for the t-product (i.e. for \(L\) being the DFT), while the proof 
works for the tensor-tensor product defined by any invertible linear transformation \(L\).

\begin{theorem}[{\cite[Theorem 4.6]{kilmer2013third}}]\label{thm:RN}
  Let \(\mathcal{A} \in \mathbb{C}_p^{l \times m}\) be a tubal matrix and \(\mathcal{A} = \mathcal{U}*\mathcal{S}*\mathcal{V}^H\) its t-SVD. 
  Let \(\mathbf{s}_i = \mathcal{S}(i,i)\), \(i=1,\ldots,n\), where \(n=\min(l,m)\).
  Suppose that \(\mathbf{s}_1,\ldots,\mathbf{s}_j\) are invertible,
  \(\mathbf{s}_{j+1},\ldots,\mathbf{s}_{j+k}\) are non-zero but not invertible, 
  and the remaining \(\mathbf{s}_{j+k+1},\ldots,\mathbf{s}_{n}\) are zero.
  Then the range \(R(\mathcal{A})\) of \(\mathcal{A}\)  and the kernel \(N(\mathcal{A})\) of \(\mathcal{A}\) are given by 
  \[R(\mathcal{A})=\{\mathcal{U}(1,:)*\mathbf{c}_1+\cdots+\mathcal{U}(j+k,:)*\mathbf{c}_{j+k} \, | \, \mathbf{c}_i=\mathbf{s}_i*\mathbf{d}_i, \mathbf{d}_i \in \mathbb{C}_p, j+1 \le i \le j+k\} ,\]
  \[N(\mathcal{A})=\{\mathcal{V}(j+1,:)*\mathbf{c}_{j+1}+\cdots+\mathcal{V}(m,:)*\mathbf{c}_m \, | \, \mathbf{s}_i*\mathbf{c}_i=\mathbf{0}, j+1 \le i \le j+k\}.\]
\end{theorem}

\section{The Small-t Transpose for Third-Order Tensors}\label{sec:smallt}

We first introduce the following transpose, and then 
we extend some familiar notions and results in matrix algebra (such as 
the Kronecker product) to tubal matrices based on this transpose.

\begin{definition}
  Let \(\mathcal{A} \in \mathbb{C}_p^{I \times J}\) be a tubal matrix.
  We denote by \(\mathcal{A}^t\) the tubal matrix obtained by simply transposing the frontal slices of \(\mathcal{A}\).
  Thus \(\mathcal{A}^t\) is obtained by transposing the tubal matrix \(\mathcal{A}\), i.e., \(\mathcal{A}^t(j,i)=\mathcal{A}(i,j)\) for all \(i \in I, j\in J\).
\end{definition}

The transpose defined above will be referred to as the small-t transpose (or the face-wise transpose) 
in contrast to the capital-T transpose introduced in the previous section. A series of basic properties concerning the small-t transpose will be established in the sequel.

\begin{proposition}\label{prop:smallttrans}
  The small-t transpose enjoys the rule \((\mathcal{A}*\mathcal{B})^t = \mathcal{B}^t*\mathcal{A}^t\) where \(\mathcal{A} \in \mathbb{C}_p^{I \times J}, \mathcal{B} \in \mathbb{C}_p^{J \times K}\).
\end{proposition}
  \begin{proof}
  \[\begin{aligned}
    (\mathcal{A}*\mathcal{B})^t(k,i)
    =&(\mathcal{A}*\mathcal{B})(i,k) \\
    \overset{(\ref{eq:matrixproduct})}{=}& \sum_{j} \mathcal{A}(i,j)*\mathcal{B}(j,k)\\
    =& \sum_j \mathcal{A}^t(j,i)*\mathcal{B}^t(k,j) \\
    =& \sum_j \mathcal{B}^t(k,j)*\mathcal{A}^t(j,i) \\
    \overset{(\ref{eq:matrixproduct})}{=}& (\mathcal{B}^t*\mathcal{A}^t) (k,i),\\
    \end{aligned}\]
  where the fourth equality holds since the tensor-tensor product of tubal scalars are commutative.
  \end{proof}

We shall make use of the following proposition in establishing our main theorem.

\begin{proposition}\label{prop:tunitary}
  \(\mathcal{A} \in \mathbb{C}_p^{I \times I}\) is unitary if and only if \(\mathcal{A}^t\) is.
\end{proposition}

\begin{proof}
By definition, \(\mathcal{A}\) is unitary if
\(\mathcal{A}*\mathcal{A}^{H}=\mathcal{A}^{H}*\mathcal{A}=\mathcal{I}\), which is equivalent to
\[L(\mathcal{A}*\mathcal{A}^{H})=L(\mathcal{A}^{H}*\mathcal{A})=L(\mathcal{I}).\]
Note that 
\[\left(L(\mathcal{A}*\mathcal{A}^{H})\right)^{(i)}=
L(\mathcal{A})^{(i)} L(\mathcal{A}^H)^{(i)}
\overset{\ref{def:herm}}{=}
L(\mathcal{A})^{(i)}  \left(L(\mathcal{A})^{(i)}\right)^H,\]
\[\left(L(\mathcal{A}^{H}*\mathcal{A})\right)^{(i)}= 
L(\mathcal{A}^H)^{(i)} L(\mathcal{A})^{(i)} 
\overset{\ref{def:herm}}{=}
\left(L(\mathcal{A})^{(i)}\right)^H L(\mathcal{A})^{(i)}.\]
Thus \(\mathcal{A}\) is unitary if and only if for \(i=1,\ldots,p\),
\[L(\mathcal{A})^{(i)}  \left(L(\mathcal{A})^{(i)}\right)^H= \left(L(\mathcal{A})^{(i)}\right)^H L(\mathcal{A})^{(i)} 
= \left(L(\mathcal{I})\right)^{(i)}=I,\]
where \(I \in \mathbb{C}^{I \times I}\) is the identity matrix.
That is,
\(\mathcal{A}\) is unitary if and only if the frontal slices of \(L(\mathcal{A})\) are unitary matrices.

Observe that 
\begin{equation}\label{eqlat}
  L(\mathcal{A}^t)=L(\mathcal{A})^t, {\rm\,\,i.e.,\,\,} L(\mathcal{A}^t)^{(k)}=(L(\mathcal{A})^{(k)})^T,\,\,k=1,\ldots,p,
\end{equation} since for any \(j \in J\), \(i \in I\),
  \[
  \begin{aligned}
    L(\mathcal{A}^t)(j,i)
  =&L\left(\mathcal{A}^t(j,i)\right)\\
  =&L\left(\mathcal{A}(i,j)\right)\\
  =&L(\mathcal{A})(i,j)\\
  =&\left(L(\mathcal{A})\right)^t(j,i),
  \end{aligned}  
  \]
where the first and third equalities hold by definition of the operator \(L\).
Therefore, we have for \(i=1,\ldots,p,\)
  \[
    \begin{aligned}
      L(\mathcal{A}^t)^{(i)}  \left(L(\mathcal{A}^t)^{(i)}\right)^{H} 
      \overset{(\ref{eqlat})}{=} &\left(L(\mathcal{A})^{(i)}\right)^T \left(\left(L(\mathcal{A})^{(i)}\right)^T\right)^{H} \\
      = &\overline{\left(L(\mathcal{A})^{(i)}\right)^H}\,\,  \overline{L(\mathcal{A})^{(i)}} \\
      = &\overline{\left(L(\mathcal{A})^{(i)}\right)^H\,\, L(\mathcal{A})^{(i)}},\\
    \end{aligned}
  \]
  where \(\overline{L(\mathcal{A})^{(i)}}\) refers to the matrix obtained by taking the conjugate of 
  \(L(\mathcal{A})^{(i)}\) and the last equality comes from the fact that 
  \(\bar{A}\bar{B}=\overline{AB}\) for matrices \(A,B\).

  Similarly we have for \(i=1,\ldots,p,\)
  \[\left(L(\mathcal{A}^t)^{(i)}\right)^{H} L(\mathcal{A}^t)^{(i)}  
  =\overline{ L(\mathcal{A})^{(i)}\,\,\left(L(\mathcal{A})^{(i)}\right)^H}.\]

  Now we are ready to prove the proposition.

If \(\mathcal{A}\) is unitary, then  for \(i=1,\ldots,p,\)
  \[L(\mathcal{A}^t)^{(i)}  \left(L(\mathcal{A}^t)^{(i)}\right)^{H}
  = \overline{\left(L(\mathcal{A})^{(i)}\right)^H\,\, L(\mathcal{A})^{(i)}}
  =\overline{I}=I,\]
  and 
  \[\left(L(\mathcal{A}^t)^{(i)}\right)^{H}L(\mathcal{A}^t)^{(i)}  
  = \overline{L(\mathcal{A})^{(i)}\,\,\left(L(\mathcal{A})^{(i)}\right)^H}
  =\overline{I}=I,\]
  so \(\mathcal{A}^t\) is unitary.

  Conversely, if \(\mathcal{A}^t\) is unitary, then for \(i=1,\ldots,p,\)
  \[\left(L(\mathcal{A})^{(i)}\right)^H\,\, L(\mathcal{A})^{(i)}
  =\overline{L(\mathcal{A}^t)^{(i)}  \left(L(\mathcal{A}^t)^{(i)}\right)^{H}}
  =\overline{I}=I,\]
  and
  \[L(\mathcal{A})^{(i)}\,\,\left(L(\mathcal{A})^{(i)}\right)^H
  =\overline{\left(L(\mathcal{A}^t)^{(i)}\right)^{H}L(\mathcal{A}^t)^{(i)}}
  =\overline{I}=I,\]
  so \(\mathcal{A}\) is unitary.
\end{proof}

Next we extend some familiar notions from matrices to tubal matrices.

\begin{definition}
  Let \(\mathcal{A}=(\mathbf{a}_{ij}) \in \mathbb{C}_p^{I \times J}\) and \(\mathcal{B}=(\mathbf{b}_{kl}) \in \mathbb{C}_p^{K \times L}\)be tubal matrices.
  The (tensor-tensor product based) Kronecker product of \(\mathcal{A}\) and \(\mathcal{B}\) is defined to be the tubal matrix 
  \(\mathcal{A} \ot \mathcal{B} \in \mathbb{C}_p^{IK \times JL}\) whose \(((i,k),(j,l))\)-th entry is the tubal scalar
  \(\mathbf{a}_{ij}*\mathbf{b}_{kl}\).
  That is,
  \[\mathcal{A} \ot \mathcal{B} = (\mathbf{a}_{i,j}*\mathcal{B})_{(i,j)\in I \times J},\]
  where \(\mathbf{a}_{i,j}*\mathcal{B}=(\mathbf{a}_{i,j}*\mathbf{b}_{k,l})_{(k,l)\in K \times L}\) is the tubal matrix obtained 
  by taking the tensor-tensor product of \(\mathbf{a}_{i,j}\) and each entry of \(\mathcal{B}\).
  Equivalently, the \((i+(k-1)K,j+(l-1)L)\)-th entry of \(\mathcal{A} \ot \mathcal{B}\) is \(\mathbf{a}_{ij}*\mathbf{b}_{kl}\).

  In general, for tubal matrices \(\mathcal{A}_{N}\in \mathbb{C}_p^{I_N \times J_N}, \cdots, \mathcal{A}_{2} \in \mathbb{C}_p^{I_2 \times J_2},\mathcal{A}_{1}\in \mathbb{C}_p^{I_1 \times J_1}\), 
  we define their Kronecker product 
  \[\mathcal{A}_{N} \ot \mathcal{A}_{N-1} \ot \cdots \mathcal{A}_{2}\ot \mathcal{A}_{1} \in \mathbb{C}_p^{I_N \cdots I_2 I_1 \times J_N \cdots J_2 J_1}\]
  to be the tubal matrix of which the \(((i_N,\ldots,i_2,i_1),(j_N,\ldots,j_2,j_1))\)-th entry
  (i.e., the \((i_1+(i_2-1)I_1+\cdots+(i_N-1)I_1\cdots I_{N-1},j_1+(j_2-1)J_1+\cdots+(j_N-1)J_1\cdots J_{N-1})\)-th entry) is
\begin{equation}\label{index_kro}
  \mathcal{A}_{N}(i_N,j_N)*\cdots *\mathcal{A}_{2}(i_2,j_2) * \mathcal{A}_{1}(i_1,j_1).
\end{equation}
\end{definition}

We need the following properties of the Kronecker product for tubal matrices.

\begin{proposition}\label{prop:kro}
  Let \(\mathcal{A}, \mathcal{B}, \mathcal{C}, \mathcal{D}\) be tubal matrices over \(\mathbb{C}_p\) of appropriate size.
  
  {\rm (i)} \((\mathcal{A} \ot \mathcal{B})^t = \mathcal{A}^t \ot \mathcal{B}^t\). 
  
  {\rm (ii)} \((\mathcal{A} \ot \mathcal{B})^H = \mathcal{A}^H \ot \mathcal{B}^H\).

  {\rm (iii)} \((\mathcal{A} \ot \mathcal{B})*(\mathcal{C} \ot \mathcal{D}) = (\mathcal{A}*\mathcal{C})\ot (\mathcal{B}*\mathcal{D})\).

  {\rm (iv)} \((\mathcal{A} \ot \mathcal{B}) \ot \mathcal{C} = \mathcal{A} \ot (\mathcal{B} \ot \mathcal{C})\).

  {\rm (v)} If \,\(\mathcal{U}_1 \in \mathbb{C}_p^{I\times I}, \mathcal{U}_2  \in \mathbb{C}_p^{J\times J}\) are unitary, then 
  \(\mathcal{U}_1 \ot \mathcal{U}_2\) is also unitary.
\end{proposition}

\begin{proof}
  {\rm (i)} Let \(\mathcal{A} \in \mathbb{C}_p^{I \times J}\), \(\mathcal{B} \in \mathbb{C}_p^{K \times L}\).
  Then \(\mathcal{A} \ot \mathcal{B} \in \mathbb{C}_p^{IK \times JL}\),
  \((\mathcal{A} \ot \mathcal{B})^t \in \mathbb{C}_p^{JL \times IK}\).
   For \(i \in I, j \in J, k \in K, l \in L \),
    \[
    \begin{aligned}
     (\mathcal{A} \ot \mathcal{B})^t((j,l),(i,k))
     &=(\mathcal{A} \ot \mathcal{B})((i,k),(j,l))\\
     &=\mathcal{A}(i,j)*\mathcal{B}(k,l)\\
     &=\mathcal{A}^t(j,i)*\mathcal{B}^t(l,k)\\
     &=(\mathcal{A}^t \ot \mathcal{B}^t)((j,l),(i,k)).     
    \end{aligned}
    \]

  {\rm (ii)}  
    We introduce an auxiliary notation
    used only in this proof.

    For tubal matrices \(\mathcal{A} \in \mathbb{C}_p^{I \times J}\) and \(\mathcal{B} \in \mathbb{C}_p^{K \times L}\) 
    we define \(\mathcal{A} \oh \mathcal{B} \in \mathbb{C}_p^{IK \times JL}\) (with ``H'' referring to ``Hadamard'') 
    to be the tubal matrix whose 
    \(\left((i,k),(j,l)\right)\)-th entry is \( \mathcal{A}(i,j) \odot \mathcal{B}(k,l)\), where \(\odot\) is the Hadamard product.
In other words, for \(i=1,\ldots,p,\) we have
\begin{equation}\label{eqH}
  \left(\mathcal{A} \oh \mathcal{B}\right)^{(i)} = \mathcal{A}^{(i)} \otimes \mathcal{B}^{(i)},
\end{equation}
where \(\mathcal{A}^{(i)} \otimes \mathcal{B}^{(i)}\) refers to the Kronecker product in the usual sense between matrices $\mathcal A^{(i)}$ and $\mathcal B^{(i)}$.

Note that 
    \begin{equation}\label{eq1}L(\mathcal{A} \ot \mathcal{B}) = L(\mathcal{A}) \oh L(\mathcal{B}),\end{equation}
   since 
   the \(\left((i,k),(j,l)\right)\)-th entry of \(L(\mathcal{A} \ot \mathcal{B})\) is 
   \[\begin{aligned}
     L\left(\mathcal{A}(i,j)*\mathcal{B}(k,l)\right)
     \overset{(\ref{def:tproductubalscalar})}{=}&L\left(\mathcal{A}(i,j)\right) \odot L\left(\mathcal{B}(k,l)\right)\\
     =&L(\mathcal{A})(i,j) \odot L(\mathcal{B})(k,l),
   \end{aligned}\]
   which is also the \(\left((i,k),(j,l)\right)\)-th entry of \(L(\mathcal{A}) \oh L(\mathcal{B})\).

The formula for Hermitian transpose then follows, since for \(i=1,\ldots,p,\)
  \[
    \begin{aligned}
      \left(L\left((\mathcal{A} \ot \mathcal{B})^H\right)\right)^{(i)}
    \overset{(\ref{def:herm})}{=}&\left(\left(L(\mathcal{A} \ot \mathcal{B})\right)^{(i)}\right)^H\\
      =&\overline{\left(\left(L(\mathcal{A} \ot \mathcal{B})\right)^{(i)}\right)^T}\\
      \overset{(\ref{eq1})}{=}&\overline{\left(\left(L(\mathcal{A}) \oh L(\mathcal{B})\right)^{(i)}\right)^T}      \\
      \overset{\cref{eqH}}{=}&\overline{\left(L(\mathcal{A})^{(i)} \otimes L(\mathcal{B})^{(i)}\right)^T}\\
      =&\overline{\left(L(\mathcal{A})^{(i)}\right)^T \otimes \left(L(\mathcal{B})^{(i)}\right)^T}\\
      =&\overline{\left(L(\mathcal{A})^{(i)}\right)^T} \otimes \overline{\left(L(\mathcal{B})^{(i)}\right)^T}\\
      =&\left(L(\mathcal{A})^{(i)}\right)^H \otimes \left(L(\mathcal{B})^{(i)}\right)^H\\
      \overset{(\ref{def:herm})}{=}&L(\mathcal{A}^H)^{(i)} \otimes L(\mathcal{B}^H)^{(i)}\\
      \overset{\cref{eqH}}{=}&\left(L(\mathcal{A}^H) \oh L(\mathcal{B}^H)\right)^{(i)}  \\
      \overset{\cref{eq1}}{=}&\left(L(\mathcal{A}^H \ot \mathcal{B}^H)\right)^{(i)},
    \end{aligned}
  \]
which means \(L\left((\mathcal{A} \ot \mathcal{B})^H\right)=L\left(\mathcal{A}^H \ot \mathcal{B}^H\right)\), implying that
\((\mathcal{A} \ot \mathcal{B})^H= \mathcal{A}^H \ot \mathcal{B}^H.\) 
In the derivation we used properties of the Kronecker product for matrices.

  {\rm (iii)} Suppose \(\mathcal{A} \in \mathbb{C}_p^{I \times J}, \mathcal{B} \in \mathbb{C}_p^{L \times M},
  \mathcal{C} \in \mathbb{C}_p^{J\times K},\mathcal{D} \in \mathbb{C}_p^{M \times N}\).
  As block tubal matrices, we have
  \[
    \begin{aligned}
      \left((\mathcal{A} \ot \mathcal{B})*(\mathcal{C} \ot \mathcal{D})\right)(i,k) 
      &= \sum_j \left(\mathcal{A}(i,j)*\mathcal{B}\right)*\left(\mathcal{C}(j,k)*\mathcal{D}\right)\\
      &= \sum _j \left(\mathcal{A}(i,j)*\mathcal{C}(j,k)\right) *(\mathcal{B}*\mathcal{D})\\
      &=(\mathcal{A}*\mathcal{C})(i,k)*(\mathcal{B}*\mathcal{D})\\
      &= \left((\mathcal{A}*\mathcal{C})\ot (\mathcal{B}*\mathcal{D})\right)(i,k),
    \end{aligned}
  \]
where the second equality comes from the associativity and commutativity of the tensor-tensor product for tubal scalars.

  {\rm (iv)} Suppose \(\mathcal{A} \in \mathbb{C}_p^{I \times J}, \mathcal{B} \in \mathbb{C}_p^{K \times L},
  \mathcal{C} \in \mathbb{C}_p^{M \times N}\).
  Both \((\mathcal{A} \ot \mathcal{B}) \ot \mathcal{C}\)
  and \(\mathcal{A} \ot (\mathcal{B} \ot \mathcal{C})\) 
  are tubal matrices of size \(IKM \times JLN\) whose \(((i,k,m),(j,l,n))\)-th entry is
\[\mathcal{A}(i,j)*\mathcal{B}(k,l)*\mathcal{C}(m,n).\]

{\rm (v)} Since \(\mathcal{U}_1 \in \mathbb{C}_p^{I\times I}, \mathcal{U}_2  \in \mathbb{C}_p^{J\times J}\) are unitary, we have
\[\begin{aligned}
(\mathcal{U}_1 \ot \mathcal{U}_2)^H * (\mathcal{U}_1 \ot \mathcal{U}_2)
\overset{({\ref{prop:kro}{,~{\rm ii}}})}{=}&(\mathcal{U}_1^H \ot \mathcal{U}_2^H)*(\mathcal{U}_1 \ot \mathcal{U}_2)\\
\overset{({\ref{prop:kro}{,~{\rm iii}}})}{=}&(\mathcal{U}_1^H \ot \mathcal{U}_1)*(\mathcal{U}_2^H \ot \mathcal{U}_2)\\
=&\mathcal{I}_I \ot \mathcal{I}_J\\
=&\mathcal{I}_{IJ}.
\end{aligned}\]
\end{proof}

\begin{definition}
  A tubal tensor of order \(N\) with entries in \(\mathbb{C}_p\) is an element \(\mathcal{A} \in \mathbb{C}_p^{I_1 \times I_2 \times \cdots \times I_N}\),
  that is, a multi-way array \((\mathbf a_{i_1i_2 \cdots i_N})_{(i_1,i_2, \cdots, i_N) \in I_1 \times I_2 \times \cdots \times I_N}\), where
  \(\mathbf a_{i_1i_2 \cdots i_N} \in \mathbb{C}_p\) are tubal scalars of length \(p\).
  In particular, a tubal tensor of order \(2\) is a tubal matrix.
\end{definition}

Since \(\mathbb{C}_p\) is isomorphic to \(\mathbb{C}^p\), \(\mathbb{C}_p^{I_1 \times I_2 \times \cdots \times I_N}\) is isomorphic 
to \(\mathbb{C}^{I_1 \times I_2 \times \cdots \times I_N \times p}\).
Thus a tubal tensor of order \(N\) in \(\mathbb{C}_p^{I_1 \times I_2 \times \cdots \times I_N}\) is essentially a tensor of order \(N+1\) in 
\(\mathbb{C}^{I_1 \times I_2 \times \cdots \times I_{N} \times p}\).
Conversely, every tensor of order \(N+1\) in \(\mathbb{C}^{I_1 \times I_2 \times \cdots \times I_{N+1}}\) can be regarded as a tubal 
tensor of order \(N\) in \(\mathbb{C}_{I_{N+1}}^{I_1 \times I_2 \times \cdots \times I_{N}}\). 
In particular, third-order tensors can be identified as tubal matrices. 

We can unfold a tubal tensor into a tubal matrix. For the mode-\(n\) unfolding we follow the convention of \cite{kolda2010tensor}.

\begin{definition}\label{def:unfold}
  Let \(\mathcal{A} \in \mathbb{C}_p^{I_1 \times I_2 \times \cdots \times I_N}\) be a tubal tensor of order \(N\) 
  with entries in \(\mathbb{C}_p\).
  We define its mode-\(n\) unfolding (where \(1 \le n \le N\)) to be the tubal matrix 
  \(\mathcal{A}_{(n)} \in \mathbb{C}_p^{I_n \times (I_1 I_2 \cdots I_{n-1}  I_{n+1} \cdots  I_N)}\) 
  such that the \((i_1,\ldots,i_N)\)-th tubal tensor entry maps to the \((i_n,j)\)-th tubal matrix entry where
  \[j = 1+ \sum^{N}_{k=1,k\neq n}((i_k-1)\prod^{k-1}_{m=1,m\neq n}I_m).\] 
\end{definition}
In particular, for the mode-\(1\) unfolding, the \((i_1,\ldots,i_N)\)-th entry of \(\mathcal{A}\) is mapped to 
the \((i_1,j)\)-th entry of \(\mathcal{A}_{(1)}\), where
\begin{equation}\label{eq:mode1}
  j= i_2 + (i_3 - 1)I_2 + (i_4 - 1)I_2I_3 + \cdots + (i_N - 1)I_2I_3 \cdots I_{N-1}.
\end{equation}

Next we introduce the notion of mode-\(n\) t-rank of a tubal tensor.

\begin{definition}
  Let \(\mathcal{A} \in \mathbb{C}_p^{I_1 \times I_2 \times \cdots \times I_N}\) be a tubal tensor of order \(N\) and \(\mathcal{A}_{(n)}\)
  its mode-\(n\) unfolding. Then the t-rank of the tubal matrix \(\mathcal{A}_{(n)}\) is called the mode-\(n\) t-rank of 
  \(\mathcal{A}\).
\end{definition}

Finally we define the \(n\)-mode product based on the
tensor-tensor product and prove a proposition that connects the \(n\)-mode product with 
the tensor-tensor product of tubal matrix unfoldings.

\begin{definition}
  The \(n\)-mode product of a tubal tensor \(\mathcal{A} \in \mathbb{C}_p^{I_1 \times \cdots \times I_N}\) by a tubal matrix 
  \(\mathcal{U} \in \mathbb{C}_p^{J \times I_n}\) is the tubal tensor \(\mathcal{A}*_n \mathcal{U} \in \mathbb{C}_p^{I_1 \times \cdots \times I_{n-1} \times J \times I_{n+1} \times \cdots \times I_N} \) 
  obtained by taking the tensor-tensor product of \(\mathcal{U}\) and
  the mode-\(n\) tubal vectors of \(\mathcal{A}\). Entrywisely, we have
  \[(\mathcal{A}*_n \mathcal{U})(i_1,\ldots, i_{n-1}, j,  i_{n+1 },\ldots, i_N) = \sum_{i_n=1}^{I_n}\mathcal{A}(i_1,\ldots, i_n, \cdots i_N)*
  \mathcal{U}(j,i_n).\]
\end{definition}

From the above definition, \(\mathcal{A}*_n \mathcal{U} = \mathcal{B}\) is the same as \(\mathcal{U}*\mathcal{A}_{(n)}=\mathcal{B}_{(n)}\).

\begin{proposition}
  Let \(\mathcal{A} \in \mathbb{C}_p^{I_1 \times \cdots \times I_N}\).
  The following properties of the \(n\)-mode product based on tensor-tensor product hold.

  {\rm (i)} \((\mathcal{A}*_n \mathcal{F})*_m \mathcal{G} =(\mathcal{A}*_m \mathcal{G})*_n \mathcal{F}\) if \(m \neq n\).

  {\rm (ii)} \((\mathcal{A}*_n \mathcal{F})*_n \mathcal{G} = \mathcal{A}*_n(\mathcal{G}*\mathcal{F}).\)
\end{proposition}

\begin{proof}
  {\rm (i)} Let \(m>n\).
  For the \((i_1,\ldots,i_{n-1},j,i_{n+1},\ldots,i_{m-1}, k,i_{m+1},\ldots,i_N)\)-th entry, we have
  \[
  \begin{aligned}
    &\sum_{i_m}^{I_m} \left(\left(\sum_{i_n=1}^{I_n}\mathcal{A}(i_1,\ldots,i_N)*\mathcal{F}(j,i_n)
    \right)*\mathcal{G}(k,i_m)\right)\\
  =&\sum_{i_n}^{I_n} \left(\left(\sum_{i_m=1}^{I_m}\mathcal{A}(i_1,\ldots,i_N)*\mathcal{G}(k,i_m)\right)*\mathcal{F}(j,i_n)\right),
  \end{aligned}  
  \]
where the equality comes from the associativity and commutativity of the tensor-tensor product of tubal scalars.

  {\rm (ii)} Let \(\mathcal{F}\in \mathbb{C}_p^{J \times I_n}, \mathcal{G} \in \mathbb{C}_p^{K \times J}\).
  For the \((i_1,\ldots,i_{n-1},k,i_{n+1},\ldots,i_N)\)-th entry, we have
  \[
    \begin{aligned}
      &\sum_{j}^{J} \left\{\left(\sum_{i_n=1}^{I_n}\mathcal{A}(i_1,\ldots,i_N)*\mathcal{F}(j,i_n)\right)*\mathcal{G}(k,j)\right\}\\
      =& \sum_{i_n} \left\{\mathcal{A}(i_1,\ldots,i_N)*\left(\sum_{j}\mathcal{F}(j,i_n)*\mathcal{G}(k,j)\right)\right\}\\
      =&\sum_{i_n} \left\{\mathcal{A}(i_1,\ldots,i_N)*\left(\sum_{j}\mathcal{G}(k,j)*\mathcal{F}(j,i_n)\right)\right\}.
    \end{aligned}
  \]
where the equalities come from the associativity and commutativity of the tensor-tensor product of tubal scalars respectively.
\end{proof}

The following proposition is the most crucial one in deriving the existence of Hot-SVD for tubal tensors.

\begin{proposition}[Tubal matrix representation of \(n\)-mode product based on tensor-tensor product]\label{prop:matrep}
  Let \(\mathcal{A}, \mathcal{S} \in \mathbb{C}_p^{I_1 \times \cdots \times I_N}\),
  \(\mathcal{U}_{1} \in \mathbb{C}_p^{I_1\times I_1},\ldots,\mathcal{U}_{N}\in \mathbb{C}_p^{I_N\times I_N}\).
  Then
  \[\mathcal{A} = \mathcal{S}*_1 \mathcal{U}_{1} *_2 \mathcal{U}_{2} *_3 \cdots *_N \mathcal{U}_{N}\] 
  is equivalent to 
  \[ \mathcal{A}_{(n)} = \mathcal{U}_{n}*\mathcal{S}_{(n)}*\left(\mathcal{U}_{N} \ot \cdots \ot \mathcal{U}_{n+1} \ot \mathcal{U}_{n-1} \ot \cdots \ot \mathcal{U}_{1}\right)^t.\]
\end{proposition}

\begin{proof}

 First of all observe that we can reduce to the case \(n=1\).

 To see this, for an arbitrary \(n\), define \(\widetilde{\mathcal{A}} \in \mathbb{C}_p^{I_n \times I_1 \times I_2 \times \cdots \times I_{n-1} \times I_{n+1} \times \cdots \times I_N}\) by 
 \[\widetilde{\mathcal{A}}(i_n,i_1,i_2,\ldots,i_{n-1},i_{n+1},\ldots,i_N)=\mathcal{A}(i_1,i_2,\ldots,i_n,\ldots,i_N).\]
 Then
 \[\mathcal{A}_{(n)}=\widetilde{\mathcal{A}}_{(1)}.\]
 Moreover
 \[\mathcal{A}=\mathcal{S} *_1 \mathcal{U}_{1}*_2\cdots *_N \mathcal{U}_{N}\]
 is the same as
 \[\widetilde{\mathcal{A}}=\widetilde{\mathcal{S}} *_1 \mathcal{U}_{n} *_2 \mathcal{U}_{1}  \cdots *_n \mathcal{U}_{n-1} *_{n+1} 
 \mathcal{U}_{n+1} \cdots *_N \mathcal{U}_{N},\]
 and
 \[\mathcal{A}_{(n)}=\mathcal{U}_{n} * S_{(n)} * 
 \left(\mathcal{U}_{N} \ot \cdots \ot \mathcal{U}_{n+1} \ot \mathcal{U}_{n-1} \ot \cdots \ot \mathcal{U}_{1}\right)^t\]
 is the same as
 \[\widetilde{\mathcal{A}}_{(1)}=\mathcal{U}_{n} * \widetilde{\mathcal{S}}_{(1)} * \left(\mathcal{U}_{N} \ot \cdots \ot \mathcal{U}_{n+1} \ot \mathcal{U}_{n-1}\ot \cdots \ot \mathcal{U}_{1} \right)^t.\]
 Thus we have reduced the general statement to the case \(n=1\).

 Now we prove the equivalence for \(n=1\). That is,
 \[\mathcal{A} = \mathcal{S}*_1 \mathcal{U}_{1} *_2 \mathcal{U}_{2} *_3 \cdots *_N \mathcal{U}_{N}\]
 if and only if 
 \[\mathcal{A}_{(1)} = \mathcal{U}_{1}*\mathcal{S}_{(1)}*\left(\mathcal{U}_{N} \ot \mathcal{U}_{N-1} \cdots \ot \mathcal{U}_{2}\right)^t.\]

 By our convention for unfolding \cref{eq:mode1}, we have
 \begin{equation}\label{index_A}\mathcal{A}_{(1)}(i_1,j)=\mathcal{A}(i_1,i_2,\ldots,i_N),\end{equation}
 where \begin{equation}\label{index_j}
  j= i_2 + (i_3 - 1)I_2 + (i_4 - 1)I_2I_3 + \cdots + (i_N - 1)I_2I_3 \cdots I_{N-1}.
 \end{equation}
 Similarly we have 
 \begin{equation}\label{index_S}\mathcal{S}_{(1)}(i'_1,j') = \mathcal{S}(i'_1,i'_2,\ldots,i'_N),\end{equation}
 where \begin{equation}\label{index_j'}j'= i'_2 + (i'_3 - 1)I_2 + (i'_4 - 1)I_2I_3+\cdots + (i'_N - 1)I_2I_3\cdots I_{N-1}.
 \end{equation}

 Define 
 \[\mathcal{V}=\left(\mathcal{U}_{N} \ot \cdots \ot \mathcal{U}_{2}\right)^t
 \overset{(\ref{prop:kro}{,~{\rm i}})}{=}
 {\mathcal{U}_{N}}^t \ot \cdots \ot {\mathcal{U}_{2}}^t \]
 and 
 \[\mathcal{W}=\mathcal{U}_{1}*\mathcal{S}_{(1)} * \mathcal{V}.\]

According to the definition of Kronecker product \cref{index_kro}, 
\(\mathcal{V}={\mathcal{U}_{N}}^t \ot \cdots \ot {\mathcal{U}_{2}}^t \in \mathbb{C}_{p}^{(I_NI_{N-1}\cdots I_2)\times (I_NI_{N-1}\cdots I_2)}\),
and the tubal scalar \({\mathcal{U}_{N}}^t(i'_N,i_N)*\cdots*{\mathcal{U}_{2}}^t(i'_2,i_2)\) is
the \[\Bigl(i'_2+(i'_3-1)I_2+\cdots+(i'_N-1)I_2\cdots I_{N-1},i_2 + (i_3 - 1)I_2 + \cdots + (i_N - 1)I_2I_3 \cdots I_{N-1}
\Bigr)\text{-th}\]
entry of \(\mathcal{V}\), which, according to \cref{index_j,index_j'}, is exactly the \((j',j)\)-th entry of \(\mathcal{V}\).
Therefore we have
\begin{equation}\label{index_V}
 \begin{aligned}
  \quad \mathcal{V}(j',j)
  =&\left({\mathcal{U}_{N}}^t \ot \cdots \ot {\mathcal{U}_{2}}^t\right) (j',j)
  =\,\,{\mathcal{U}_{N}}^t(i'_N,i_N)*\cdots*{\mathcal{U}_{2}}^t(i'_2,i_2).
 \end{aligned}  
\end{equation}
  
 Then we have
 \begin{equation}\label{index_W}
   \begin{aligned}
      & \quad \mathcal{W}(i_1,j)\\
      =&\left(\mathcal{U}_{1}*\mathcal{S}_{(1)} * \mathcal{V}\right)(i_1,j)\\
      =&\sum_{i'_1,j'} \mathcal{U}_{1}(i_1,i'_1) * \mathcal{S}_{(1)}(i'_1,j')* \mathcal{V}(j',j)\\
      =&\sum_{i'_1,i'_2,\ldots,i'_N} \mathcal{U}_{1}(i_1,i'_1) * \mathcal{S}(i'_1,i'_2,\ldots,i'_N)*
      \mathcal{U}_{N}^t(i'_N,i_N)*\cdots*
      \mathcal{U}_{2}^t(i'_2,i_2)\\
      =&\sum_{i'_1,i'_2,\ldots,i'_N} \mathcal{U}_{1}(i_1,i'_1) * \mathcal{S}(i'_1,i'_2,\ldots,i'_N) * \mathcal{U}_{N}(i_N,i'_N) * \cdots * \mathcal{U}_{2}(i_2,i'_2)\\
      =&\sum_{i'_1,\ldots,i'_N} \mathcal{S}(i'_1,\ldots,i'_N) * \mathcal{U}_{1}(i_1,i'_1) * \cdots *\mathcal{U}_{N}(i_N,i'_N),
   \end{aligned}
\end{equation}
where the third equality holds due to \cref{index_S,index_V} and the last equality comes from the commutativity of the tensor-tensor
product of tubal scalars.

Now we are ready to prove that \(\mathcal{A}=\mathcal{S} *_1 \mathcal{U}_{1}*_2 \cdots *_N \mathcal{U}_{N}\)
 if and only if \(\mathcal{A}_{(1)}=\mathcal{W}=\mathcal{U}_{1}*\mathcal{S}_{(1)} * \mathcal{V}\).

 If \(\mathcal{A}=\mathcal{S} *_1 \mathcal{U}_{1}*_2 \cdots *_N \mathcal{U}_{N}\), this means that, entrywisely,
 \[\mathcal{A}(i_1,\ldots,i_N)=\sum_{i'_1,\ldots,i'_N} \mathcal{S}(i'_1,\ldots,i'_N) * \mathcal{U}_{1}(i_1,i'_1) * \cdots *\mathcal{U}_{N}(i_N,i'_N);\]
 then, according to \cref{index_W},
 \[\mathcal{W}(i_1,j)=\mathcal{A}(i_1,\ldots,i_N)=\mathcal{A}_{(1)}(i_1,j),\]
where \(j\) is as defined in \cref{index_j}. Therefore we have \(\mathcal{A}_{(1)}=\mathcal{W}\).

Conversely, if \(\mathcal{W}=\mathcal{A}_{(1)}\), then
 \[
   \begin{aligned}
    &\quad\mathcal{A}(i_1,\ldots,i_N)\\
    =&\quad \mathcal{A}_{(1)}(i_1,j)\\
    =&\quad \mathcal{W}(i_1,j)\\
    =&\sum_{i'_1,\ldots,i'_N}\mathcal{S}(i'_1,\ldots,i'_N)*\mathcal{U}_{1}(i_1,i'_1)* \cdots * \mathcal{U}_{N}(i_N,i'_N),
   \end{aligned}
 \]
where \(j\) is as defined in \cref{index_j} and the last equality is due to \cref{index_W}.
Therefore we have \(\mathcal{A} = \mathcal{S}*_1 \mathcal{U}_{1} *_2 \mathcal{U}_{2} *_3 \cdots *_N \mathcal{U}_{N}\).
\end{proof}

\section{Higher-order t-SVD of tubal tensors}
\label{sec:main}

With these preparations, we are ready to establish the Hot-SVD of tubal tensors. 
To facilitate the comparison with   HOSVD of (usual) tensors, we quote the following theorem.

\begin{theorem}[HOSVD of tensors, {\cite[Theorem 2]{DeLDeMV}}]
  Every tensor \(\mathcal{A} \in \mathbb{C}^{I_1 \times \cdots \times I_N}\) can be written as the product
  \[\mathcal{A} = \mathcal{S}\times_1 U_{1} \cdots \times_N U_{N},\]
  where

  {\rm 1}. \(U_{n}\) is a unitary matrix for \(n=1,\ldots,N\),

  {\rm 2}. \(\mathcal{S} \in \mathbb{C}^{I_1 \times \cdots \times I_N}\) has the following properties:

  {\rm (i)} all-orthogonality: 
  
  \[\Braket{\mathcal{S}_{i_n=\alpha},\mathcal{S}_{i_n=\beta}} = 0\] for 
  \(1 \le \alpha \neq \beta \le I_n\), where \(\mathcal{S}_{i_n=\alpha}\) is the \((N-1)\)-th order   tensor obatained by fixing the \(n\)-th 
  index of \(\mathcal{S}\) to be \(\alpha\).

  {\rm (ii)} ordering: 
  
  \[\|\mathcal{S}_{i_n=1}\| \ge \|\mathcal{S}_{i_n=2}\| \ge \cdots \ge \|\mathcal{S}_{i_n=I_n}\|\] for all possible
  values of \(n\), where \(\|\mathcal{S}_{i_n=\alpha}\|\) is the same as the Frobenius norm of the \(\alpha\)-th row vector of the mode-\(n\)
  unfolding matrix
  \(\mathcal{S}_{(n)}\).

\end{theorem}

Now we state and prove the tubal version of the above theorem, based on tensor-tensor product.

\begin{theorem}[Hot-SVD of tubal tensors]\label{thm:Hot-SVD}
  Every tubal tensor \(\mathcal{A} \in \mathbb{C}_p^{I_1 \times \cdots \times I_N}\) can be written as the product
  \[ \mathcal{A} = \mathcal{S}*_1 \mathcal{U}_{1} \cdots *_N \mathcal{U}_{N},\]
  where 

  {\rm 1}. \(\mathcal{U}_{n} \in \mathbb{C}_p^{I_n \times I_n}\) is a unitary tubal matrix for \(n=1,\ldots,N\),

  {\rm 2}. \(\mathcal{S} \in \mathbb{C}_p^{I_1 \times \cdots \times I_N}\) has the following properties:

    {\rm (i)} all-orthogonality: for all \(1 \le n \le N\) and \(1 \le \alpha \neq \beta \le I_n\),
    
    \[\sum \mathcal{S}(i_1,\ldots,i_{n-1},\alpha,i_{n+1},\ldots,i_N)^H*\mathcal{S}(i_1,\ldots,i_{n-1},\beta,i_{n+1},\ldots,i_N) = \mathbf{0},\]
    where the sum is taken over \(i_1,\ldots,i_{n-1},i_{n+1},\ldots,i_{N}\) and \(\mathbf{0}=(0,\ldots,0) \in \mathbb{C}_p\). 

    {\rm (ii)} ordering: for \(L=cW\) where \(W\) is a unitary transformation and \(c \in \mathbb{C}\) is a non-zero scalar, we have
    further    
    \[\|\mathcal{S}_{i_n=1}\| \ge \|\mathcal{S}_{i_n=2}\| \ge \cdots \ge \|\mathcal{S}_{i_n=I_n}\|\] for all possible
    values of \(n\), where \(\|\mathcal{S}_{i_n=\alpha}\|\) is the same as the Frobenius norm of the \(\alpha\)-th row tubal
    vector of the mode-\(n\)
    unfolding tubal matrix
    \(\mathcal{S}_{(n)}\).
\end{theorem}

Decomposition of the form in \cref{thm:Hot-SVD} is called higher-order t-SVD (Hot-SVD) in this work.

\begin{proof}
  According to   \cref{prop:matrep},
  we only need to prove that the equivalent tubal matrix representation
  \[\mathcal{A}_{(n)} = \mathcal{U}_{n}*\mathcal{S}_{(n)}*\left(\mathcal{U}_{N} \ot \cdots \ot \mathcal{U}_{n+1} \ot \mathcal{U}_{n-1} \ot \cdots \ot \mathcal{U}_{1}\right)^t\]
  holds for all \(n\).

  For \(n=1,\ldots,N\), let \( \mathcal{A}_{(n)} = \mathcal{U}_{n} * \Sigma_{n} * \mathcal{V}_{n}^{H}\) be a t-SVD of the tubal matrix 
  \(\mathcal{A}_{(n)}\). 
  Define 
  \[\mathcal{S} = \mathcal{A} *_1 \mathcal{U}_{1}^H 
  \cdots *_N \mathcal{U}_{N}^H.\]
  
  Then by \cref{prop:matrep},

  \[\mathcal{S}_{(n)} = \mathcal{U}_{n}^H * \mathcal{A}_{(n)}* \left(\mathcal{U}_{N}^{H} \ot \cdots \ot \mathcal{U}_{n+1}^{H} 
  \ot \mathcal{U}_{n-1}^{H} \cdots \ot \mathcal{U}_{1}^{H}\right)^t .\]
    Multiplying both sides by \(\mathcal{U}_{n}\) and \(\left(\mathcal{U}_{N} \ot \cdots \ot \mathcal{U}_{n+1} \ot \mathcal{U}_{n-1} \ot \cdots \ot \mathcal{U}_{1}\right)^t\),
  we have

\[ 
  \begin{aligned}
    &\mathcal{U}_{n}*\mathcal{S}_{(n)}*\left(\mathcal{U}_{N} \ot \cdots \ot \mathcal{U}_{n+1} \ot \mathcal{U}_{n-1} \ot \cdots \ot \mathcal{U}_{1}\right)^t\\
    =&\mathcal{U}_{n}*\mathcal{U}_{n}^{H} * \mathcal{A}_{(n)}* \left(\mathcal{U}_{N}^{H} \ot \cdots \ot \mathcal{U}_{n+1}^{H} \ot \mathcal{U}_{n-1}^{H} \ot \cdots \ot \mathcal{U}_{1}^{H}\right)^t\\
    & * \left(\mathcal{U}_{N} \ot \cdots \ot \mathcal{U}_{n+1} \ot \mathcal{U}_{n-1} \ot \cdots \ot \mathcal{U}_{1}\right)^t\\
    \overset{(\ref{prop:kro}{,~{\rm i}})}{=}& \mathcal{I}_{I_n}*\mathcal{A}_{(n)}*\left(\mathcal{U}_{N}^{Ht} 
    \ot \cdots \ot \mathcal{U}_{n+1}^{Ht} \ot \mathcal{U}_{n-1}^{Ht} \ot \cdots \ot \mathcal{U}_{1}^{Ht}\right)\\
    & * \left(\mathcal{U}_{N}^{t} \ot \cdots \ot \mathcal{U}_{n+1}^{t} \ot \mathcal{U}_{n-1}^{t} \ot \cdots \ot \mathcal{U}_{1}^{t}\right)\\
    \overset{(\ref{prop:kro}{,~{\rm iii}})}{=}& \mathcal{I}_{I_n}*\mathcal{A}_{(n)}* \left(\left(\mathcal{U}_{N}^{Ht}*\mathcal{U}_{N}^{t}\right) 
    \ot \cdots \ot \left(\mathcal{U}_{n+1}^{Ht}*\mathcal{U}_{n+1}^{t}\right) \right.\\
    & \left. \ot \left(\mathcal{U}_{n-1}^{Ht}*\mathcal{U}_{n-1}^{t}\right) \ot \cdots \ot \left(\mathcal{U}_{1}^{Ht}*\mathcal{U}_{1}^{t}\right)\right)\\
    \overset{(\ref{prop:smallttrans})}{=}& \mathcal{I}_{I_n}*\mathcal{A}_{(n)}*\left(\left(\mathcal{U}_{N}*\mathcal{U}_{N}^{H}\right)^t 
    \ot \cdots \ot \left(\mathcal{U}_{n+1}*\mathcal{U}_{n+1}^{H}\right)^t\right.\\
    & \ot\left. \left(\mathcal{U}_{n-1}*\mathcal{U}_{n-1}^{H}\right)^t \ot \cdots \ot \left(\mathcal{U}_{1}*\mathcal{U}_{1}^{H}\right)^t\right)\\
    =& \mathcal{I}_{I_n}* \mathcal{A}_{(n)}* \left(\mathcal{I}_{I_N}\ot \cdots \ot \mathcal{I}_{I_{n+1}} \ot \mathcal{I}_{I_{n-1}} \ot \cdots \ot \mathcal{I}_{I_1}\right)\\
    =&\mathcal{A}_{(n)},
  \end{aligned}  
\]
where the tubal matrices \(\mathcal{U}_{n}\)'s are unitary since they come from t-SVD.

  From
  \[\mathcal{A}_{(n)} = \mathcal{U}_{n} * \Sigma_{n} *
  \mathcal{V}_{n}^{H}\] 
  and
  \[\mathcal{A}_{(n)}=\mathcal{U}_{n}*\mathcal{S}_{(n)}*\left(\mathcal{U}_{N} \ot \cdots \ot \mathcal{U}_{n+1} \ot \mathcal{U}_{n-1} \ot \cdots \ot \mathcal{U}_{1}\right)^t,\]
  we obatain
  \[\mathcal{U}_{n}*\mathcal{S}_{(n)}*\left(\mathcal{U}_{N} \ot \cdots \ot \mathcal{U}_{n+1} \ot \mathcal{U}_{n-1} \ot \cdots \ot \mathcal{U}_{1}\right)^t
  =\mathcal{U}_{n} * \Sigma_{n} * \mathcal{V}_{n}^{H}.\]
  Multiplying both sides by \(\mathcal{U}_{n}^{H}\) and 
  \(\left(\mathcal{U}_{N} \ot \cdots \ot \mathcal{U}_{n+1} \ot \mathcal{U}_{n-1} \ot \cdots \ot \mathcal{U}_{1}\right)^{tH}\),
  we have
  \[
  \begin{aligned}
    \mathcal{S}_{(n)} 
    =& \mathcal{U}_{n}^{H}*\mathcal{U}_{n}* \Sigma_{n} * \mathcal{V}_{n}^{H} 
    *\left(\mathcal{U}_{N} \ot \cdots \ot \mathcal{U}_{n+1} \ot \mathcal{U}_{n-1} \ot \cdots \ot \mathcal{U}_{1}\right)^{tH}\\
  =&\Sigma_{n} * \mathcal{V}_{n}^{H} *(\mathcal{U}_{N} \ot \cdots \ot 
  \mathcal{U}_{n+1} \ot \mathcal{U}_{n-1} \ot \cdots \ot \mathcal{U}_{1})^{tH}\\
  \overset{(\ref{prop:kro}{,{\rm i}})}{=}& \Sigma_{n} * \mathcal{V}_{n}^{H} *\left(\mathcal{U}^{(N)t} \ot \cdots \ot 
  \mathcal{U}^{(n+1)t} \ot \mathcal{U}^{(n-1)t} \ot \cdots \ot \mathcal{U}^{(1)t}\right)^H.
  \end{aligned}  
  \]
  Since the tubal matrices \(\mathcal{U}_{i}^{t}, i=1,\ldots,N\) are unitary (by \cref{prop:tunitary}), the tubal matrix 
  \[\left(\mathcal{U}_{N}^{t} \ot \cdots \ot 
  \mathcal{U}_{n+1}^{t} \ot \mathcal{U}_{n-1}^{t} \ot \cdots \ot \mathcal{U}_{1}^{t}\right)^H\] is also unitary by 
  \cref{prop:kro}{,~{\rm (v)}}. 
  Then the tubal matrix \(\mathcal{W}=\mathcal{V}_{n}^{H} * \left(\mathcal{U}_{N}^{t} \ot \cdots \ot 
  \mathcal{U}_{n+1}^{t} \ot \mathcal{U}_{n-1}^{t} \ot \cdots \ot \mathcal{U}_{1}^{t}\right)^H\) is unitary. 
  As \(\Sigma_{n}\)
  is an f-diagonal tubal matrix, we conclude that the row tubal vectors of \(\mathcal{S}_{(n)}\),
  being tubal scalar multiples of the row tubal vectors of \(\mathcal{W}\), are orthogonal to each other,
  whence the all orthogonality of \(\mathcal{S}\).

  For the ordering property, observe first that
  \(\left\|\mathcal{S}_{i_n=\alpha}\right\| = \left\|\Sigma_{n}(\alpha,\alpha)\right\|\), which holds when \(L=cW\) (see \cite[Theorem 3.1]{kilmer2021tensor}).
  From the construction of t-SVD, each frontal slice of \(L(\Sigma_{n})\) has nonincreasing
  singular values, implying that the Frobenius norms \(\|L(\Sigma_{n})(\alpha,\alpha)\|\) are nonincreasing and consequently 
  the Frobenius norms \[\left\|\Sigma_{n}(\alpha,\alpha)\right\|=\left\|L^{-1}\left(L(\Sigma_{n})(\alpha,\alpha)\right)\right\|\]
  are nonincreasing 
  \[\left\|\Sigma_{n}(1,1)\right\| \ge \left\|\Sigma_{n}(2,2)\right\| \ge \cdots \ge \left\|\Sigma_{n}(I_N,I_N)\right\|,\]
  since \(\left\|L^{-1}(\mathbf{a})\right\| = \left\|c^{-1}W^H(\mathbf{a})\right\|=\left\|c^{-1}\right\|\cdot\left\|W^H(\mathbf{a})\right\|
  =\left\|c^{-1}\right\|\cdot\left\|\mathbf{a}\right\|\) for any tubal scalar \(\mathbf{a} \in \mathbb{C}_p\).
  Then 
  \[\left\|\mathcal{S}_{i_n=1}\right\| \ge \left\|\mathcal{S}_{i_n=2}\right\| \ge \cdots \ge \left\|\mathcal{S}_{i_n=I_n}\right\|,\]
  since \(\left\|\mathcal{S}_{i_n=\alpha}\right\| = \left\|\Sigma_{n}(\alpha,\alpha)\right\|\) as noted above.
\end{proof}

\begin{remark}
The above proof tells us that \cref{prop:matrep} plays a key role to show the validness of Hot-SVD, and  we emphasize that the link in \cref{prop:matrep} between  the Hot-SVD of a tubal tensor \(\mathcal{A}\) and the t-SVD 
  of the unfoldings 
  \(\mathcal{A}_{(n)}\) of \(\mathcal{A}\) does not hold if we replace the small-t transpose with the usual capital-T transpose. 
\end{remark}

The above proof actually indicates how   Hot-SVD of a given tubal tensor can be computed: the tubal matrix \(\mathcal{U}_{n}\) can be directly found through
the t-SVD of the unfolding tubal matrix \(A_{(n)}\), and the core tubal tensor \(\mathcal{S}\) can be computed by the \(n\)-mode product. 
When \(L\) is the DFT, the computational
complexity of this procedure  is of the same order as that of t-SVD for higher-order tensors defined in \cite{martin2013order}.

Many properties of HOSVD have clear counterparts in our model based on the tensor-tensor product.

\begin{property}[generalization]
  The Hot-SVD of a tubal matrix boils down to the t-SVD.
\end{property}

This is obvious from \cref{prop:matrep}.

\begin{property}[\(n\)-rank]
  The mode-\(n\) tubal  rank of \(\mathcal{A}\) is the same as the highest index \(r_n\) for which \(\|\mathcal{S}_{i_n=r_n}\|>0\).
\end{property}

\begin{proof}
 We use the notations from the proof of \cref{thm:Hot-SVD}.

 By definition, the mode-\(n\) tubal rank of \(\mathcal{A}\) is number of non-zero tubal scalars on the diagonal of \(\Sigma_{n}\).
 From the proof of \cref{thm:Hot-SVD}, we have
 \[\mathcal{S}_{(n)}=\Sigma_{n} * \mathcal{V}_{n}^{H} *\left(\mathcal{U}_{N}^{t} \ot \cdots \ot 
 \mathcal{U}_{n+1}^{t} \ot \mathcal{U}_{n-1}^{t} \ot \cdots \ot \mathcal{U}_{1}^{t}\right)^H,\]
 where \(\mathcal{V}_{n}^{H}\) and \(\left(\mathcal{U}_{N}^{t} \ot \cdots \ot 
 \mathcal{U}_{n+1}^{t} \ot \mathcal{U}_{n-1}^{t} \ot \cdots \ot \mathcal{U}_{1}^{t}\right)^H\) are unitary tubal matrices.
 Thus the number of non-zero row vectors of the tubal matrix \(\mathcal{S}_{(n)}\) (which is the same as the highest index
 \(r_n\) for which \(\left\|\mathcal{S}_{i_n=r_n}\right\|>0\)) is the same as the number of non-zero tubal scalars on the diagonal
 of \(\Sigma_{n}\), i.e., the mode-\(n\) tubal rank of \(\mathcal{A}\).
\end{proof}

\begin{property}[link between Hot-SVD and t-SVD]\label{prop:link}
  The Hot-SVD gives a thin t-SVD of \(\mathcal{A}_{(n)}\) by normalizing \(\mathcal{S}_{(n)}\) to extract a diagonal tubal matrix.
\end{property}

\begin{proof}
  If \(\mathcal{A} =\mathcal{S}*_1 \mathcal{U}_{1}\cdots*_N \mathcal{U}_{N}\), then
\[\mathcal{A}_{(n)}=\mathcal{U}_{n}*\mathcal{S}_{(n)}*\left(\mathcal{U}_{N}\ot\cdots\ot \mathcal{U}_{n+1}\ot \mathcal{U}_{n-1}\ot\cdots\ot \mathcal{U}_{1}\right)^t,\]
where \(\mathcal{S}_{(n)} \) has mutually orthogonal tubal rows, with Frobenius-norms \(\sigma_1^{(n)},\ldots,\sigma_{I_n}^{(n)}\). 

Define 
\[\Sigma_{n}=\diag(\sigma_1^{(n)}\mathbf{1},\sigma_2^{(n)}\mathbf{1},\ldots,\sigma_{(I_n)}^{(n)}\mathbf{1}),\]
where \(\mathbf{1}=L^{-1}\left((1,1,\ldots,1)\right)\in \mathbb{C}_p\) is the identity tubal scalar in \(\mathbb{C}_p\).

Let \(\widetilde{\mathcal{S}}_{n}\) be the normalized version of \(\mathcal{S}_{(n)}\), i.e., \(\mathcal{S}_{(n)}=\Sigma_{n}*\widetilde{\mathcal{S}}_{(n)}\).
Define 
\[\mathcal{V}_{n}^H=\widetilde{\mathcal{S}}_{n}*\left(\mathcal{U}_{N}\ot\cdots\ot \mathcal{U}_{n+1}\ot \mathcal{U}_{n-1}\ot\cdots\ot \mathcal{U}_{1}\right)^t,\]
which is a tubal matrix with orthogonal tubal rows.
Then \(\mathcal{A}_{(n)}=\mathcal{U}_{n}*\Sigma_{n}*\mathcal{V}_{n}^H\) is a thin t-SVD of \(\mathcal{A}_{(n)}\).
\end{proof}

\begin{property}[structure]
  Keep the notation in the proof of \cref{prop:link}. That is, 
  \(\mathcal{A}_{(n)}=\mathcal{U}_{n}*\Sigma_{n}*\mathcal{V}_{n}^H\) is the t-SVD of \(\mathcal{A}_{(n)}\) obatained from the Hot-SVD of \(\mathcal{A}\).
  Denote the tubal scalars on the diagonal of \(\Sigma_{n}\) by \(\mathbf{s}_i\). Suppose that the conditions in \cref{thm:RN} hold. Then

  {\rm (i)} The range of \(\mathcal{A}_{(n)}\) is

  \[\left\{\mathcal{U}_{n}(1,:)*\mathbf{c}_1+\cdots+\mathcal{U}_{n}(j+k,:)*\mathbf{c}_{j+k} \, | \, 
  \mathbf{c}_i=\mathbf{s}_i*\mathbf{d}_i, \mathbf{d}_i \in \mathbb{C}_p, j+1 \le i \le j+k\right\}.\]

  {\rm (ii)} The kernel of \(\mathcal{A}_{(n)}\) is

  \[\left\{\mathcal{V}_{n}(j+1,:)*\mathbf{c}_{j+1}+\cdots+\mathcal{V}_{n}(m,:)*\mathbf{c}_m \, | \, 
  \mathbf{s}_i*\mathbf{c}_i=\mathbf{0}, \mathbf{c}_i \in \mathbb{C}_p, j+1 \le i \le j+k\right\}.\]
\end{property}

This is exactly \cref{thm:RN} as applied to the tubal matrix \(\mathcal{A}_{(n)}\).

\begin{property}[norm] When \(L=cW\), where \(W\) is a unitary transformation and \(c\in \mathbb{C}\) is non-zero scalar, we have
  \[\|\mathcal{A}\| = \|\mathcal{S}\|.\]
\end{property}

This follows from the unitary invariance of the Frobenius norm of tubal matrices \cite[Theorem 3.1]{kilmer2021tensor}.

\section{ Truncated Hot-SVD and Sequentially Truncated Hot-SVD}
\label{sec:prop}

This section derives the truncated Hot-SVD and sequentially truncated Hot-SVD, which generalize those of \cite{de2000a,vannieuwenhoven2012new} to the tubal setting. The algorithms are depicted in Algorithms \ref{alg:tr_hotsvd} and \ref{alg:seq_tr_hotsvd}.

\begin{algorithm}[t]
  \caption{tr-Hot-SVD\((\mathcal{A},I_{1}^{\prime},\ldots,
  I_{N}^{\prime})\)}
  \label{alg:tr_hotsvd}
  \begin{algorithmic}
    
    \WHILE{\(1 \le n \le N\)}
    \STATE{\(\mathcal{U}_{n} \leftarrow I_{n^{\prime}}\) leading left singular tubal vectors of \(\mathcal{A}_{(n)}\)}
    \ENDWHILE
    \STATE{\(\mathcal{S} \leftarrow \mathcal{A}*_1 \mathcal{U}^{(1)H}*_2 \mathcal{U}^{(2)H} \cdots *_N \mathcal{U}^{(N)H}\) }

    \RETURN $\mathcal{S}, \mathcal{U}_{1},\mathcal{U}_{2},\ldots,\mathcal{U}_{N}$

  \end{algorithmic}
\end{algorithm}

\begin{algorithm}
  \caption{Seq-tr-Hot-SVD\((\mathcal{A},I'_1,I'_2,\ldots,I'_N)\)}
  \label{alg:seq_tr_hotsvd}
  \begin{algorithmic}
    \STATE{\(\widehat{\mathcal{S}} \leftarrow \mathcal{A}\)}
    \WHILE{\(1 \le n \le N\)}
    \STATE{\(\widehat{\mathcal{U}}_{n} \leftarrow\) the \(I'_n\) leading left singular tubal vectors of
    \(\widehat{\mathcal{S}}_{(n)}\)}
    \STATE{\(\widehat{\mathcal{S}} \leftarrow \mathcal{A}*_1 \widehat{\mathcal{U}}_{1}^{H} \cdots *_{n} \widehat{\mathcal{U}}_{n}^{H}\)}
    \ENDWHILE
    \RETURN $\widehat{\mathcal{S}}, \widehat{\mathcal{U}}_{1},\widehat{\mathcal{U}}_{2},\ldots,\widehat{\mathcal{U}}_{N}$
  \end{algorithmic}
\end{algorithm}

To prove an error bound for truncated Hot-SVD and sequentially truncated Hot-SVD, we first need some technical preparations.

\begin{proposition}\label{prop:orthoganalityinfrobeniusnorm}
  Suppose \(L=cW\), where \(c \in \mathbb{C}\) is a non-zero scalar and \(W\) is a unitary transformation.
  Let \(\mathcal{A}_{1},\cdots,\mathcal{A}_{N} \in \mathbb{C}_p^{I_1 \times I_2}\) be tubal matrices that are orthogonal 
  to each other with respect to the tensor-tensor product: for \(1\le m \neq n \le N\),
  \[\mathcal{A}_{m}^{H}*\mathcal{A}_{n}=\mathcal{O}.\]
  Then \(\mathcal{A}_{1},\cdots,\mathcal{A}_{N}\) are orthogonal in the Frobenius norm: for \(1\le m \neq n \le N\),
  \[\braket{\mathcal{A}_m,\mathcal{A}_n}=0.\]
  Consequently we have
  \[\|\mathcal{A}_{1}\|^2+\|\mathcal{A}_{2}\|^2+\cdots+\|\mathcal{A}_{N}\|^2
  =\|\mathcal{A}_{1}+\mathcal{A}_{2}+\cdots+\mathcal{A}_{N}\|^2.\]
\end{proposition}
\begin{proof}
  Since \(\mathcal{A}_{m}^{H}*\mathcal{A}_{n}=\mathcal{O}\), we have for \(k=1,\ldots,p,\)
  \[L(\mathcal{A}_m)^{(k)H} L(\mathcal{A}_n)^{(k)}=0,\]
where $0$ above denotes the zero matrix of the proper size.  This implies that 
  \begin{equation}\label{eq:Lortho}
    \braket{L(\mathcal{A}_m)^{(k)},L(\mathcal{A}_n)^{(k)}} = \mathrm{tr} \left(L(\mathcal{A}_m)^{(k)H} L(\mathcal{A}_n)^{(k)}\right)=0.
  \end{equation}
  Therefore the matrices \(L(\mathcal{A}_1)^{(k)},\ldots,L(\mathcal{A}_N)^{(k)}\) are orthogonal to each other in the Frobenius norm. 
  Then it follows that
  \begin{equation}
    \begin{aligned}
      &\braket{\mathcal{A}_m,\mathcal{A}_n}\\
      =&\sum_{i_1,i_2,k}\mathcal{A}_m(i_1,i_2)^{(k)}\mathcal{A}_n(i_1,i_2)^{(k)}\\
      =&\sum_{i_1,i_2}\left(\sum_{k=1}^{p}\mathcal{A}_m(i_1,i_2)^{(k)}\mathcal{A}_n(i_1,i_2)^{(k)}\right)\\
      =&\sum_{i_1,i_2}\left\langle\mathcal{A}_m(i_1,i_2),\mathcal{A}_n(i_1,i_2)\right\rangle\\
      =&\sum_{i_1,i_2}\frac{1}{|c|^2}\left\langle L\left(\mathcal{A}_m(i_1,i_2)\right),L\left(\mathcal{A}_n(i_1,i_2)\right)\right\rangle\\
      =&\sum_{i_1,i_2}\frac{1}{|c|^2}\left\langle L\left(\mathcal{A}_m\right)(i_1,i_2),L\left(\mathcal{A}_n\right)(i_1,i_2)\right\rangle\\
      =&\frac{1}{|c|^2}\sum_{i_1,i_2,k} \left(L\left(\mathcal{A}_m\right)(i_1,i_2)\right)^{(k)} 
      \left(L\left(\mathcal{A}_n\right)(i_1,i_2)\right)^{(k)}\\
      =&\frac{1}{|c|^2}\sum_{k=1}^p\left(\sum_{i_1,i_2}\left(L\left(\mathcal{A}_m\right)(i_1,i_2)\right)^{(k)} 
      \left(L\left(\mathcal{A}_n\right)(i_1,i_2)\right)^{(k)}\right)\\
      \overset{\ref{eq:Lortho}}{=}&\frac{1}{|c|^2}\sum_{k=1}^{p} \left\langle L(\mathcal{A}_m)^{(k)},L(\mathcal{A}_n)^{(k)} \right\rangle\\
      =&0.
    \end{aligned}
  \end{equation}
  Finally the equality concerning the Frobenius norms follows from the bilinearity of the inner product.
\end{proof}

\begin{proposition}\label{prop:normdecrease}
  Suppose \(L=cW\), where \(c \in \mathbb{C}\) is a non-zero scalar and \(W\) is a unitary transformation.
  Let \(\mathcal{A} \in \mathbb{C}_p^{I_1 \times \cdots \times I_N}\) be a tubal tensor and 
  \(\breve{\mathcal{U}}_{n}\in \mathbb{C}_p^{I_n \times R_n}\) (\(I_n \ge R_n\)) be a partially unitary 
  tubal matrix (i.e. \(\breve{\mathcal{U}}_{n}^{H}*\breve{\mathcal{U}}_{n}=\mathcal{I}\)).
  Then 
  \[\|\mathcal{A}*_n\left(\breve{\mathcal{U}}_{n}* \breve{\mathcal{U}}_{n}^{H}\right)\|^2+
    \|\mathcal{A}*_n\left(\mathcal{I}-\breve{\mathcal{U}}_{n}* \breve{\mathcal{U}}_{n}^{H}\right)\|^2
  = \|\mathcal{A}\|^2 .\]
  In particular, \(\|\mathcal{A}*_n\left(\breve{\mathcal{U}}^{(n)} *\breve{\mathcal{U}}_{n}^{H}\right)\| \le \|\mathcal{A}\|\).
\end{proposition}
\begin{proof}
  Note that \(\left(\breve{\mathcal{U}}_{n}* \breve{\mathcal{U}}_{n}^{H}\right)^{H}
  *\left(\mathcal{I}-\breve{\mathcal{U}}_{n}* \breve{\mathcal{U}}_{n}^{H}\right)
  =\breve{\mathcal{U}}_{n}* \breve{\mathcal{U}}_{n}^{H}-\breve{\mathcal{U}}_{n} *\breve{\mathcal{U}}_{n}^{H}*
  \breve{\mathcal{U}}_{n}* \breve{\mathcal{U}}_{n}^{H}
  =\mathcal{O},\) so \cref{prop:orthoganalityinfrobeniusnorm}
  implies that 
  \[\begin{aligned}
    &\|\mathcal{A}*_n\left(\breve{\mathcal{U}}_{n}* \breve{\mathcal{U}}_{n}^{H}\right)\|^2+
  \|\mathcal{A}*_n\left(\mathcal{I}-\breve{\mathcal{U}}_{n}* \breve{\mathcal{U}}_{n}^{H}\right)\|^2\\
  =&\|\left(\breve{\mathcal{U}}_{n}* \breve{\mathcal{U}}_{n}^{H}\right)*\mathcal{A}_{(n)}\|^2
  +\|\left(\mathcal{I}-\breve{\mathcal{U}}_{n}* \breve{\mathcal{U}}_{n}^{H}\right)*\mathcal{A}_{(n)}\|^2\\
  =&\|\mathcal{A}_{(n)}\|^2=\|\mathcal{A}\|^2.
  \end{aligned}\] 
\end{proof}

\begin{proposition}\label{prop:error}
  Suppose \(L=cW\), where \(c \in \mathbb{C}\) is a non-zero scalar and \(W\) is a unitary transformation.
  Let \(\mathcal{A} \in \mathbb{C}_p^{I_1 \times I_2 \times \cdots \times I_N}\) be a tubal tensor of order \(N\).
  Let \(\mathcal{A}\) be approximated by 
  \[\breve{\mathcal{A}} = \mathcal{A} *_1 \left(\breve{\mathcal{U}}_{1}* \breve{\mathcal{U}}_{1}^{H}\right)\cdots
  *_N\left( \breve{\mathcal{U}}_{N} *\breve{\mathcal{U}}_{N}^{H} \right),\]
  where \(\breve{\mathcal{U}}_{n} \in \mathbb{C}_p^{I_n \times R_n}\) (\(I_n \ge R_n\), \(n=1,\ldots,N\)) are partially unitary tubal matrices (i.e. 
  \(\breve{\mathcal{U}}_{n}^{H}*\breve{\mathcal{U}}_{n}=\mathcal{I}\)).
  Then the squared approximation error is 
  \[
  \begin{aligned}  
  &\quad\,\|\mathcal{A}-\breve{\mathcal{A}}\|^2\\
  =&\quad\,\|\mathcal{A}*_1\left(\mathcal{I}-\breve{\mathcal{U}}_{1}* \breve{\mathcal{U}}_{1}^{H}\right)\|^2\\
  &+\|\mathcal{A}*_1\left(\breve{\mathcal{U}}_{1} *\breve{\mathcal{U}}_{1}^{H}\right) 
  *_2 \left(\mathcal{I}-\breve{\mathcal{U}}_{2} *\breve{\mathcal{U}}_{2}^{H}\right)\|^2\\
  &+\|\mathcal{A}*_1\left(\breve{\mathcal{U}}_{1}* \breve{\mathcal{U}}_{1}^{H}\right)
  *_2\left(\breve{\mathcal{U}}_{2}* \breve{\mathcal{U}}_{2}^{H}\right)
  *_3\left(\mathcal{I} - \breve{\mathcal{U}}_{3} *\breve{\mathcal{U}}_{3}^{H}\right)\|^2\\
  &+\cdots\\
  &+\|\mathcal{A}*_1\left(\breve{\mathcal{U}}_{1}* \breve{\mathcal{U}}_{1}^{H}\right)
  \cdots
  *_{N-1}\left(\breve{\mathcal{U}}_{N-1}* \breve{\mathcal{U}}_{N-1}^{H}\right)
  *_N\left(\mathcal{I}- \breve{\mathcal{U}}_{N}* \breve{\mathcal{U}}_{N}^{H}\right)\|^2
  \end{aligned}
  \]
\end{proposition}
\begin{proof}
Note that
  \begin{equation}\label{eq:telescope}
    \begin{aligned}
      \mathcal{A}-\breve{\mathcal{A}}\,\,\,
      =&\quad \,\,\mathcal{A}
      -\mathcal{A}*_1\left(\breve{\mathcal{U}}_{1}* \breve{\mathcal{U}}_{1}^{H}\right)
      +\mathcal{A}*_1\left(\breve{\mathcal{U}}_{1}* \breve{\mathcal{U}}_{1}^{H}\right)\\
      &-\mathcal{A}*_1\left(\breve{\mathcal{U}}_{1}* \breve{\mathcal{U}}_{1}^{H}\right)
      *_2\left(\breve{\mathcal{U}}_{2}* \breve{\mathcal{U}}_{2}^{H}\right)\\
      &+\mathcal{A}*_1\left(\breve{\mathcal{U}}_{1}* \breve{\mathcal{U}}_{1}^{H}\right)
      *_2\left(\breve{\mathcal{U}}_{2}* \breve{\mathcal{U}}_{2}^{H}\right)\\
      &\cdots\\
      &-\mathcal{A}*_1 \left(\breve{\mathcal{U}}_{1}* \breve{\mathcal{U}}_{1}^{H}\right)\cdots
      *_{N-1}\left( \breve{\mathcal{U}}_{N-1}* \breve{\mathcal{U}}_{N-1}^{H} \right)\\
      &+\mathcal{A}*_1 \left(\breve{\mathcal{U}}_{1}* \breve{\mathcal{U}}_{1}^{H}\right)\cdots
      *_{N-1}\left( \breve{\mathcal{U}}_{N-1}* \breve{\mathcal{U}}_{N-1}^{H} \right)\\
      &-\mathcal{A}*_1 \left(\breve{\mathcal{U}}_{1}* \breve{\mathcal{U}}_{1}^{H}\right)\cdots
      *_N\left( \breve{\mathcal{U}}_{N}* \breve{\mathcal{U}}_{N}^{H} \right)\\
      =&\quad\,\,\mathcal{A}*_1\left(\mathcal{I}-\breve{\mathcal{U}}_{1}* \breve{\mathcal{U}}_{1}^{H}\right)\\
      &+\mathcal{A}*_1\left(\breve{\mathcal{U}}_{1}* \breve{\mathcal{U}}_{1}^{H}\right)
      *_2\left(\mathcal{I}-\breve{\mathcal{U}}_{2}* \breve{\mathcal{U}}_{2}^{H}\right)\\
      &+\cdots\\
      &+\mathcal{A}*_1 \left(\breve{\mathcal{U}}_{1}* \breve{\mathcal{U}}_{1}^{H}\right)\cdots
      *_{N-1}\left( \breve{\mathcal{U}}_{N-1}* \breve{\mathcal{U}}_{N-1}^{H} \right)
      *_N\left(\mathcal{I}-\breve{\mathcal{U}}_{N}* \breve{\mathcal{U}}_{N}^{H}\right).
    \end{aligned}
  \end{equation}
  We now show that any two distinct terms in the above expression are orthogonal to each other in the Frobenius norm. 
  For \(i<j\), let 
  \(\breve{\mathcal{B}}, \breve{\mathcal{C}}\) be the \(i\)-th and \(j\)-th terms in the above 
  expression, respectively.
  Consider 
  \[\breve{\mathcal{B}}_{(i)}:=\left(\mathcal{I}-\breve{\mathcal{U}}_{i}* \breve{\mathcal{U}}_{i}^{H}\right)
  *\mathcal{A}_{(i)}*\left(\mathcal{I}\ot \cdots \ot \mathcal{I} \ot \left(\breve{\mathcal{U}}_{i-1}* \breve{\mathcal{U}}_{i-1}^{H}\right) 
  \ot \cdots \ot 
  \left(\breve{\mathcal{U}}_{1}* \breve{\mathcal{U}}_{1}^{H}\right)\right)^t\]
  and 
  \[\begin{aligned}
    \breve{\mathcal{C}}_{(i)}
  :=&\left(\breve{\mathcal{U}}_{i}* \breve{\mathcal{U}}_{i}^{H}\right)
  *\mathcal{A}_{(i)}* \Biggl( 
    \mathcal{I}\ot \cdots \ot \mathcal{I} \ot
  \left(\mathcal{I}-\breve{\mathcal{U}}_{j}* \breve{\mathcal{U}}_{j}^{H}\right)
   \ot \left(\mathcal{I}-\breve{\mathcal{U}}_{j-1} * \breve{\mathcal{U}}_{j-1}^{H}\right) \ot \cdots\\
  &\ot \left(\breve{\mathcal{U}}_{i+1}* \breve{\mathcal{U}}_{i+1}^{H}\right)
  \ot \left(\breve{\mathcal{U}}_{i-1}* \breve{\mathcal{U}}_{i-1}^{H}\right) 
  \ot \cdots \ot 
  \left(\breve{\mathcal{U}}_{1}* \breve{\mathcal{U}}_{1}^{H}\right) \Biggr)^t.
  \end{aligned}\]
  Then \(\breve{\mathcal{C}}_{(i)}^{H}*\breve{\mathcal{B}}_{(i)}=\mathcal{O},\) since 
  \(\left(\breve{\mathcal{U}}_{i}* \breve{\mathcal{U}}_{i}^{H}\right)^{H} * 
  \left(\mathcal{I}-  \breve{\mathcal{U}}_{i}* \breve{\mathcal{U}}_{i}^{H}\right) = \mathcal{O}\).
  Then \cref{prop:orthoganalityinfrobeniusnorm} implies that the \(i\)-th term and the \(j\)-th term are orthogonal in the Frobenius norm. 
\end{proof}

Now we consider the truncated Hot-SVD. 
\begin{theorem}[Error bound for tr-Hot-SVD] \label{th:error_bound_tr_hotsvd}
  Suppose \(L=cW\), where \(W:\mathbb{C}_p \to \mathbb{C}_p\) is a unitary transformation \(W\) and \(c \in \mathbb{C}\) is a non-zero scalar.
  Let \(\mathcal{A} \in \mathbb{C}_p^{I_1 \times \cdots \times I_N}\) be a tubal tensor and 
  \[ \mathcal{A} = \mathcal{S} *_1 \mathcal{U}_{1} \cdots *_N \mathcal{U}_{N}\]
  its Hot-SVD.
  Define \(\widehat{A}\) to be 
  \[ \widehat{\mathcal{A}} = \widehat{\mathcal{S}} *_1 \mathcal{U}_{1} \cdots *_N \mathcal{U}_{N},\]
  where \(\widehat{\mathcal{S}}\) is obtained by truncation of the first \(I_1'\times \cdots \times I_N'\)  tubal scalars of 
  \(\widehat{\mathcal{S}}\) (with other tubal scalars being zero).
  Then 
  \[
  \begin{aligned}
    \left\|\mathcal{A}-\widehat{\mathcal{A}}\right\| 
    &\le \sqrt{\sum_{i_1=I_1'+1}^{R_1} (\sigma_{i_1}^{(1)})^2+ \cdots +\sum_{i_N=I_N'+1}^{R_N}(\sigma_{i_N}^{(N)})^2}\\
    &\le \sqrt N\left\|\mathcal{A} - \mathcal{A}^*\right\|,
  \end{aligned}
  \]
  where \(R_n\) is the t-rank of \(\mathcal{A}_{(n)}\), \(\sigma_{i_n}^{(n)}\) is the Frobenius norm of the \(i_n\)-th row of \(\mathcal{S}_{(n)}\) (which is equal to the Frobenius norm 
  of the \(i_n\)-th tubal scalar on the diagonal of \(\Sigma_n\), the tubal
  matrix appearing in the t-SVD of \(\mathcal{A}_{(n)}\)),
  \(\mathcal{A}^*=\mathcal{S}^* *_1 U_1^* \cdots *_N U_N^*\), 
  and \((\mathcal{S}^*, U_1^*,\ldots,U_N^*)\) is the optimal solution to the following (tensor-tensor product-based)
  low-rank approximation problem:
      
  \[
    \begin{aligned}
      &\min_{\mathcal{S},\mathcal U_1,\ldots,\mathcal U_N}\left\|\mathcal{A}-\mathcal{S}*_1 \mathcal U_1 \cdots *_N \mathcal U_N\right\|,\\
      &\,\,\,\quad \mathrm{ s.t.} \quad \mathcal{U}_1^H*\mathcal{U}_1=\mathcal{I}_{I_1'}, \cdots, \mathcal{U}_N^H*\mathcal{U}_N=\mathcal{I}_{I_N'},\\
      &\,\,\,\quad \quad \quad \mathcal{U}_n \in \mathbb{C}_p^{I_n\times I_n'} \, (n=1,\ldots,N), \mathcal{S}\in \mathbb{C}_P^{I_1'\times \cdots \times I_N'}.
    \end{aligned}
  \]
\end{theorem}

\begin{proof}
Let \(\mathcal{A}_{(n)}=\mathcal{U}_n * \Sigma_n * \mathcal{V}_n^{H}\) be the t-SVD of the mode-\(n\) unfolding.
Define \(\overline{\mathcal{U}}_{n}\) to be the tubal matrix obtained from the first \(I'_n\) tubal column vectors of \(\mathcal{U}_n\).
By definition 
\[\widehat{\mathcal{A}}=\widehat{\mathcal{S}}
*_1 \mathcal{U}_1\cdots *_N \mathcal{U}_N
=\widetilde{\mathcal{S}}*_1 \overline{\mathcal{U}}_1\cdots *_N \overline{\mathcal{U}}_N,\]
where \(\widetilde{\mathcal{S}} \in \mathbb{C}_p^{I'_1
\times \cdots \times I'_N}\) is obtained by truncation of
\(\mathcal{S}\).
Note also that
\[\widetilde{\mathcal{S}}=
\mathcal{A}*_1 \overline{\mathcal{U}}_{1}^H \cdots *_N\overline{\mathcal{U}}_{N}^H,\]
since \(\mathcal{S}=
\mathcal{A}*_1 \mathcal{U}_{1}^H \cdots *_N\mathcal{U}_{N}^H.\)
Then we have 
\begin{equation}
  \begin{aligned}
    \widehat{\mathcal{A}} = &\left(\mathcal{A} *_1 \overline{\mathcal{U}}_{1}^H \cdots *_N\overline{\mathcal{U}}_{N}^H\right)
  *_1 \overline{\mathcal{U}}_{1} \cdots *_N\overline{\mathcal{U}}_{N}\\
  =&\mathcal{A}*_1 \left(\overline{\mathcal{U}}_{1}*\overline{\mathcal{U}}_{1}^H\right) \cdots 
  *_N \left( \overline{\mathcal{U}}_{N} * \overline{\mathcal{U}}_{N}^H \right).
  \end{aligned}
\end{equation}

Then by \cref{prop:error} we have 
\begin{equation}\label{eq:errortruncatedhotsvd}
  \begin{aligned}
    &\|\mathcal{A}-\widehat{\mathcal{A}}\|^2\\
    =&\|\mathcal{A}- \mathcal{A}*_1 \left(\overline{\mathcal{U}}_{1}*\overline{\mathcal{U}}_{1}^H\right) \cdots 
    *_N \left( \overline{\mathcal{U}}_{N} * \overline{\mathcal{U}}_{N}^H \right)\|^2\\
    =&\quad\,\|\mathcal{A}*_1\left(\mathcal{I}-\overline{\mathcal{U}}_{1}* \overline{\mathcal{U}}_{1}^{H}\right)\|^2\\
  &+\|\mathcal{A}*_1\left(\overline{\mathcal{U}}_{1} *\overline{\mathcal{U}}_{1}^{H}\right) 
  *_2 \left(\mathcal{I}-\overline{\mathcal{U}}_{2} *\overline{\mathcal{U}}_{2}^{H}\right)\|^2\\
  &+\|\mathcal{A}*_1\left(\overline{\mathcal{U}}_{1}* \overline{\mathcal{U}}_{1}^{H}\right)
  *_2\left(\overline{\mathcal{U}}_{2}* \overline{\mathcal{U}}_{2}^{H}\right)
  *_3\left(\mathcal{I} - \overline{\mathcal{U}}_{3} *\overline{\mathcal{U}}_{3}^{H}\right)\|^2\\
  &+\cdots\\
  &+\|\mathcal{A}*_1\left(\overline{\mathcal{U}}_{1}* \overline{\mathcal{U}}_{1}^{H}\right)
  \cdots
  *_{N-1}\left(\overline{\mathcal{U}}_{N-1}* \overline{\mathcal{U}}_{N-1}^{H}\right)
  *_N\left(\mathcal{I}- \overline{\mathcal{U}}_{N}* \overline{\mathcal{U}}_{N}^{H}\right)\|^2.\\
  \end{aligned}
\end{equation}
For the \(n\)-th term in the above expression, by \cref{prop:normdecrease}, we have 
\begin{equation}\label{eq:errordecrease}
  \begin{aligned}
  &\|\mathcal{A}*_1\left(\overline{\mathcal{U}}_{1}* \overline{\mathcal{U}}_{1}^{H}\right)
  \cdots
  *_{n-1}\left(\overline{\mathcal{U}}_{n-1}* \overline{\mathcal{U}}_{n-1}^{H}\right)
  *_n\left(\mathcal{I}- \overline{\mathcal{U}}_{n}* \overline{\mathcal{U}}_{n}^{H}\right)\|\\
  \le
  &\|\mathcal{A}*_n\left(\mathcal{I}- \overline{\mathcal{U}}_{n}* \overline{\mathcal{U}}_{n}^{H}\right)\|.
\end{aligned}
\end{equation}

Note that
  \begin{equation}\label{eq:estimate4}
    \begin{aligned}
    &\|\mathcal{A}*_n\left(\mathcal{I}-\overline{\mathcal{U}}_{n}* \overline{\mathcal{U}}_{n}^{H}\right)\|^2\\
    =&\|\left(\mathcal{I}-\overline{\mathcal{U}}_{n}* \overline{\mathcal{U}}_{n}^{H}\right)* \mathcal{A}_{(n)} \|^2\\
    =&\|\left(\mathcal{I}-\overline{\mathcal{U}}_{n}* \overline{\mathcal{U}}_{n}^{H}\right)* 
    \mathcal{U}_n *\Sigma_n *\mathcal{V}_n^{H}\|^2\\
    =&\|\left(\mathcal{U}_n-\overline{\mathcal{U}}_{n}* \overline{\mathcal{U}}_{n}^{H}*\mathcal{U}_n\right)
    *\Sigma_n *\mathcal{V}_n^{H}\|^2\\
    =&\|\left(\mathcal{U}_n-\begin{pmatrix}
      \mathcal{U}_{n}(:,1) & \cdots & \mathcal{U}_{n}(:,I'_n)
    \end{pmatrix} * \begin{pmatrix}
      \mathcal{U}_{n}(:,1)^{H} \\  \vdots \\ \mathcal{U}_{n}(:,I'_n)^{H}
    \end{pmatrix}
     * \begin{pmatrix}
      \mathcal{U}_{n}(:,1) & \cdots & \mathcal{U}_{n}(:,I_n)
    \end{pmatrix} \right)\\
    &*\Sigma_n *\mathcal{V}_n^{H}\|^2\\
    =&\|\left(\mathcal{U}_n-\begin{pmatrix}
      \mathcal{U}_{n}(:,1) & \cdots & \mathcal{U}_{n}(:,I'_n)
    \end{pmatrix}*\begin{pmatrix}
      \mathcal{I}_{I'_n} & \mathcal{O}
    \end{pmatrix}
    \right)*\Sigma_n *\mathcal{V}_n^{H}\|^2\\
    =&\|\left(
      \mathcal{U}_n- \begin{pmatrix}
        \mathcal{U}_{n}(:,1) & \mathcal{U}_{n}(:,2) & \cdots  & \mathcal{U}_{n}(:,I'_n) & \mathcal{O}
      \end{pmatrix}
    \right)*\Sigma_n *\mathcal{V}_n^{H}\|^2\\
    =&\|\sum_{i_n=I'_n+1}^{R_n}\mathcal{U}_n(:,i_n)*\Sigma_n(:,i_n) *\mathcal{V}_n(:,i_n)^{H} \|^2\\
    \overset{\ref{prop:orthoganalityinfrobeniusnorm}}{=}&\sum_{i_n=I'_n+1}^{R_n}\|\mathcal{U}_n(:,i_n)*\Sigma_n(:,i_n) *\mathcal{V}_n(:,i_n)^{H} \|^2\\
    =&\sum_{i_n=I'_n+1}^{R_n}\|\Sigma_n(:,i_n)\|^2\\
    =&\sum_{i_n=I'_n+1}^{R_n} (\sigma_{i_n}^{(n)})^2,
  \end{aligned}\end{equation}
where the second to last equality comes from the unitary invariance of the Frobenius norm under the 
tensor-tensor product.

Set \(\bar{\mathcal{A}}_{(1)} = \mathcal{U}^* * S^* * \mathcal{V}^{*H}\), where \((S^*,\mathcal{U}^*,\mathcal{V}^*)\) is the optimal solution to the following 
low-rank approximation problem:
  \[
    \begin{aligned}
      &\min_{S,\mathcal{U},\mathcal{V}} \left\|\mathcal{A}_{(1)}-\mathcal{U}*S*\mathcal{V}^T\right\|,\\
      &\,\,\,\, \mathrm{ s.t.} \quad \mathcal{U} \in \mathbb{C}_p^{I_1 \times I_1'}, S\in \mathbb{C}_p^{I_1'\times I_1'}, 
      \mathcal{V}\in \mathbb{C}_p^{(I_2\times \cdots \times I_N)\times I_1'},\\
      &\quad \quad \quad \mathcal{U}^H*\mathcal{U}=I_{I_1'}, \mathcal{V}^H*\mathcal{V}=I_{I_1'}.
    \end{aligned}
  \]

  Then the equality \(\sum_{i_1=I_1'+1}^{I_1}(\sigma_{i_1}^{(1)})^2=\|\mathcal{A}_{(1)}-\bar{\mathcal{A}}_{(1)}\|^2\)
  follows from the Eckart-Young type result for tubal matrices \cref{thm:EY}, since 
  \[\bar{\mathcal{A}}_{(1)} = \sum_{i=1}^{I_1'}\mathcal{U}_{1}(:,i)*S^{(1)}(i,i)*\mathcal{V}^{(1)}(:,i)^H\]
  is the optimal solution to the above problem.

  Similarly we have equalities

  \[
    \sum_{i_2=I_2'+1}^{I_2}(\sigma_{i_2}^{(2)})^2=\|\mathcal{A}_{(2)}-\bar{\mathcal{A}}_{(2)}\|^2,
        \ldots,
    \sum_{i_N=I_N'+1}^{I_N}(\sigma_{i_N}^{N})^2=\|\mathcal{A}_{(N)}-\bar{\mathcal{A}}_{(N)}\|^2.
\]

Therefore, we have

  \[
    \begin{aligned}
      &\left\|\mathcal{A}-\widehat{\mathcal{A}}\right\|^2 \\
      \le&\sum_{i_1=I_1'+1}^{I_1}(\sigma_{i_1}^{(1)})^2+\cdots+\sum_{i_N=I_N'+1}^{I_N}(\sigma_{i_N}^{(N)})^2\\
      =&\left\|\mathcal{A}_{(1)}-\bar{\mathcal{A}}_{(1)}\right\|^2+\cdots+\left\|\mathcal{A}_{(N)}-\bar{\mathcal{A}}_{(N)}\right\|^2\\
      \le& \left\|\mathcal{A}_{(1)}-\mathcal{A}^{*}_{(1)}\right\|^2+\cdots+\left\|\mathcal{A}_{(N)}-\mathcal{A}^{*}_{(N)}\right\|^2\\
      =&N\left\|\mathcal{A}-\mathcal{A}^*\right\|^2,
    \end{aligned}
  \]
where the first inequality comes from \cref{eq:errortruncatedhotsvd,eq:errordecrease,eq:estimate4}, the third equality comes from the unitary invariance of the Frobenius norm 
under the tensor-tensor product (\cite[Theorem 3.1]{kilmer2021tensor})
and the inequality in the second to the last line is due to the Eckart-Young theorem for tubal tensors (\cref{thm:EY}). 
\end{proof}

Note that the truncated   HOSVD of an \((N+1)\)-th order tensor \(\mathcal{A}\) leads to an error bound of the form 
\(\|\mathcal{A}-\widehat{\mathcal{A}}\| \le \sqrt{N+1}\|\mathcal{A}-\mathcal{A}^*\|\)
(where \(\widehat{\mathcal{A}}\) is obtained by the truncated   HOSVD and \(\mathcal{A}^*\) is the solution to 
the usual (i.e. not based on tensor-tensor product) orthogonal low-rank approximation problem. 

The sequentially truncated Hot-SVD can be run in the order \(p_1,\ldots,p_N\) defined by a permutaion of \(1,\ldots,N\).
In practice,   the processing order is important. For theoretical 
analysis of the error bound, we only need to consider the case \(1,\ldots,N\).

Now we are ready to prove the error bound for sequentially truncated Hot-SVD.
\begin{theorem}[Error bound for sequentially truncated Hot-SVD]\label{thm:errorboundst}
  Suppose that \(\mathcal{A} \in \mathbb{C}_p^{I_1 \times \cdots \times I_N}\) is a tubal tensor of order \(N\). Let 
  \(\mathcal{A}_{(n)} = \mathcal{U}_n*\Sigma_n*\mathcal{V}_n^{H}\) be a t-SVD of the mode-\(n\) unfolding of \(\mathcal{A}\).
  Let \(R_n\) be the t-rank of \(\mathcal{A}_{(n)}\), i.e., \(R_n\) is the number of non-zero tubal scalars on the diagonal
  of \(\Sigma_n\).
  Let \(\sigma^{(n)}_{1},\ldots,\sigma^{(n)}_{R_n}\) be the Frobenius norms of the non-zero tubal scalars on the diagonal of \(\Sigma_n\).
  Let \(\widehat{\mathcal{A}}\) be a rank-\((I'_1,\ldots,I'_N)\) sequentially truncated Hot-SVD of \(\mathcal{A}\) where 
  \(I'_1\le R_1,\ldots,I'_N \le R_N\). That is, 
  \[\widehat{\mathcal{A}} = \widehat{\mathcal{S}}*_1\widehat{\mathcal{U}}_1\cdots*_N\widehat{\mathcal{U}}_N,\]
  where \(\widehat{\mathcal{S}},\widehat{\mathcal{U}}_1,\ldots,\widehat{\mathcal{U}}_N\) are 
  as constructed in \cref{alg:seq_tr_hotsvd}.
  In other words, \(\widehat{\mathcal{U}}_n\)  is the tubal matrix formed by the first \(I'_n\)  column tubal vectors 
  of the left unitary tubal matrix of the t-SVD of 
 \[\left(\mathcal{A} *_1\widehat{\mathcal{U}}_{1}^{H}  
 \cdots
 *_{n-1}\widehat{\mathcal{U}}_{n-1}^{H}\right)_{(n)},\]
 \[\widehat{\mathcal{S}}=\mathcal{A} *_1\widehat{\mathcal{U}}_{1}^{H}  
 \cdots
 *_{N}\widehat{\mathcal{U}}_{N}^{H},\]
 and 
 \[\widehat{\mathcal{A}}=\mathcal{A} *_1\left(\widehat{\mathcal{U}}_{1} * \widehat{\mathcal{U}}_{1}^{H}\right)  
 \cdots
 *_{N}\left(\widehat{\mathcal{U}}_{N}* \widehat{\mathcal{U}}_{N}^{H}\right).\]
 Then we have the following error bound
  \[\|\mathcal{A}-\widehat{\mathcal{A}}\| 
    \le \sqrt{\sum_{i_1=I_1'+1}^{R_1} (\sigma_{i_1}^{(1)})^2+ \cdots +\sum_{i_N=I_N'+1}^{R_N}(\sigma_{i_N}^{(N)})^2} \leq \sqrt N\|\mathcal A - \mathcal A^*\|,\]
where $\mathcal A^*$ was defined in Theorem \ref{th:error_bound_tr_hotsvd}. 
\end{theorem}
\begin{proof}
  By \cref{prop:error}, we have
  \begin{equation}\label{eq:estimate1}
    \begin{aligned}  
  &\quad\,\|\mathcal{A}-\widehat{\mathcal{A}}\|^2\\
  =&\quad\,\|\mathcal{A}*_1\left(\mathcal{I}-\widehat{\mathcal{U}}_{1}* \widehat{\mathcal{U}}_{1}^{H}\right)\|^2\\
  &+\|\mathcal{A}*_1\left(\widehat{\mathcal{U}}_{1} *\widehat{\mathcal{U}}_{1}^{H}\right) 
  *_2 \left(\mathcal{I}-\widehat{\mathcal{U}}_{2} *\widehat{\mathcal{U}}_{2}^{H}\right)\|^2\\
  &+\|\mathcal{A}*_1\left(\widehat{\mathcal{U}}_{1}* \widehat{\mathcal{U}}_{1}^{H}\right)
  *_2\left(\widehat{\mathcal{U}}_{2}* \widehat{\mathcal{U}}_{2}^{H}\right)
  *_3\left(\mathcal{I} - \widehat{\mathcal{U}}_{3} *\widehat{\mathcal{U}}_{3}^{H}\right)\|^2\\
  &+\cdots\\
  &+\|\mathcal{A}*_1\left(\widehat{\mathcal{U}}_{1}* \widehat{\mathcal{U}}_{1}^{H}\right)
  \cdots
  *_{N-1}\left(\widehat{\mathcal{U}}_{N-1}* \widehat{\mathcal{U}}_{N-1}^{H}\right)
  *_N\left(\mathcal{I}- \widehat{\mathcal{U}}_{N}* \widehat{\mathcal{U}}_{N}^{H}\right)\|^2.\\
  \end{aligned}
\end{equation}
  Let \(\overline{\mathcal{U}}_n\) be the tubal matrix formed by the first \(I'_n\)  column tubal vectors
  of the left unitary tubal matrix of the t-SVD of \(\mathcal{A}_{(n)}\). That is,
  \(\overline{\mathcal{U}}_1,\ldots,\overline{\mathcal{U}}_N\) are the unitary tubal matrices obtained from 
  the truncated Hot-SVD of \(\mathcal{A}\).
  For the \(n\)-th term in the above expression, we have 
  \begin{equation}\label{eq:estimate2}
    \begin{aligned}
      &\|\mathcal{A}*_1\left(\widehat{\mathcal{U}}_{1}* \widehat{\mathcal{U}}_{1}^{H}\right)
  \cdots
  *_{n-1}\left(\widehat{\mathcal{U}}_{n-1}* \widehat{\mathcal{U}}_{n-1}^{H}\right)
  *_n\left(\mathcal{I}- \widehat{\mathcal{U}}_{n}* \widehat{\mathcal{U}}_{n}^{H}\right)\|^2\\
  =
  &\|\mathcal{A}*_1\left(\widehat{\mathcal{U}}_{1}^{H}\right)
  \cdots
  *_{n-1}\left(\widehat{\mathcal{U}}_{n-1}^{H}\right)
  *_n\left(\mathcal{I}- \widehat{\mathcal{U}}_{n}* \widehat{\mathcal{U}}_{n}^{H}\right)\|^2\\
  \le 
&\|\mathcal{A}*_1\left(\widehat{\mathcal{U}}_{1}^{H}\right)
  \cdots
  *_{n-1}\left(\widehat{\mathcal{U}}_{n-1}^{H}\right)
  *_n\left(\mathcal{I}- \overline{\mathcal{U}}_{n}* \overline{\mathcal{U}}_{n}^{H}\right)\|^2\\
  = 
&\|\mathcal{A}*_1\left(\widehat{\mathcal{U}}_{1}* \widehat{\mathcal{U}}_{1}^{H}\right)
  \cdots
  *_{n-1}\left(\widehat{\mathcal{U}}_{n-1}* \widehat{\mathcal{U}}_{n-1}^{H}\right)
  *_n\left(\mathcal{I}- \overline{\mathcal{U}}_{n}* \overline{\mathcal{U}}_{n}^{H}\right)\|^2,
    \end{aligned}
  \end{equation}
where the inequality holds due to the the Eckart-Young theorem for t-SVD (since \(\widehat{\mathcal{U}}_n\), being  
obtained from the truncation of t-SVD gives the minimum)
and the equalities hold due to the unitary invariance of the Frobenius norm under the tensor-tensor product.
By \cref{prop:normdecrease}, we have 
\begin{equation}\label{eq:estimate3}
  \begin{aligned}
    &\|\mathcal{A}*_1\left(\widehat{\mathcal{U}}_{1}* \widehat{\mathcal{U}}_{1}^{H}\right)
  \cdots
  *_{n-1}\left(\widehat{\mathcal{U}}_{n-1}* \widehat{\mathcal{U}}_{n-1}^{H}\right)
  *_n\left(\mathcal{I}- \overline{\mathcal{U}}_{n}* \overline{\mathcal{U}}_{n}^{H}\right)\|^2\\
  \le 
  &\|\mathcal{A}  *_n\left(\mathcal{I}- \overline{\mathcal{U}}_{n}* \overline{\mathcal{U}}_{n}^{H}\right)\|^2.
  \end{aligned}
\end{equation}
Note that by \cref{eq:estimate4}, we have 
\begin{equation}\label{eq:estimate5}\|\mathcal{A}*_n\left(\mathcal{I}-\overline{\mathcal{U}}_{n}* \overline{\mathcal{U}}_{n}^{H}\right)\|^2
=\sum_{i_n=I'_n+1}^{R_n} (\sigma_{i_n}^{(n)})^2.\end{equation}
Combining \cref{eq:estimate1,eq:estimate2,eq:estimate3,eq:estimate5}, we obtain the desired error bound.
\end{proof}

\begin{remark}
  The notion of the generalized HOSVD was previously introduced in \cite[Sect. 4.2]{Liao2020} and \cite[Sect. V-C]{Liao2022}, in the context of generalized tensors (the notion of generalized tensors is a generalization of that of tubal tensors, and the generalized HOSVD was called THOSVD there). However, no proof of the existence of THOSVD was given in \cite{Liao2020,Liao2022} nor did the authors 
 explore properties such as all-orthogonality and ordering of the core tubal tensor $\mathcal S$. In fact, one of our main contributions  is the rigorous proof of the existence of Hot-SVD, where the basic definitions and properties introduced and developed in Sect. \ref{sec:smallt}  play important roles in the proof. 
  Moreover, we   introduce the truncated Hot-SVD, sequentially truncated Hot-SVD, and establish their error bounds. In addition, several properties of HOSVD are also generalized to Hot-SVD.
These  were not presented in \cite{Liao2022,Liao2020}.
\end{remark}

\section{Computational Complexity}
\label{sec:comp}
In this section, assuming that \(L\) is the DFT, we 
analyze the computational complexity of Hot-SVD, tr-Hot-SVD, and seq-tr-Hot-SVD.

To obtain the Hot-SVD of \(\mathcal{A}\) we perform the following operations:

0. Regard \(\mathcal{A} \in \mathbb{R}_{p}^{I_1 \times \cdots \times I_{N}}\) as a tubal tensor of order \(N\).

1. Apply FFT to \(\mathcal{A}_{(1)}\) to get \(\widehat{\mathcal{A}}_{(1)} \), ..., and to 
\(\mathcal{A}_{(N)}\) to get \(\widehat{\mathcal{A}}_{(N)}\).
The complexity is $O(I_1\cdots I_N p\log(p))$.

2. Apply matrix SVD to the frontal slices of 
\[\widehat{\mathcal{A}}_{(1)}\in \mathbb{C}_{p}^{I_1 \times (I_2\cdots I_{N})}, \cdots, 
\widehat{\mathcal{A}}_{(N)} \in \mathbb{C}_{p}^{I_{N} \times (I_1  \cdots  I_{N-1})},\]
to get \(\widehat{\mathcal{U}}_{1},\ldots, \widehat{\mathcal{U}}_{N}\). 
Then apply IFFT to get \(\mathcal{U}_{1},\ldots,\mathcal{U}_{N}\).

3. Apply \(n\)-mode product to get the core tensor \(\mathcal{S }= \mathcal{A}*_1 \mathcal{U}_1^{T}\cdots *_{N}\mathcal{U}_N^{T}\).

If we assume that \(I_1 \le I_2\cdots I_{N},\ldots,I_{N} \le I_1\cdots I_{N-1} \),
then step 2 requires 
\begin{equation}\label{eq:cc:1}O(pI_1^2I_2\cdots I_{N}+\cdots+pI_{N}^2N_1\cdots I_{N-1})\end{equation}
flops, as for each $\mathcal U_n$ it computes $p$ SVD of size $I_n\times \prod^N_{j\neq n}I_j$ in the Fourier domain (if $\mathcal A$ is real then only half of the SVDs need to be performed) and the operations concerning IFFT are dominated. Step 3 requires 
\begin{equation*}O(pI_1^2I_2\cdots I_{N}+ \cdots +  p I_{N}^2I_1\cdots I_{N-1})\end{equation*}
flops (\cite[Fact 3]{kilmer2011factorization}).

For tr-Hot-SVD with \(I'_1 \le I_1,\ldots, I'_N \le I_N\), the  only difference is that at step 3,  the complexity is    \begin{equation}
    \label{eq:cc:2}
    O\left(pI'_1I_1(I_2I_3\cdots I_N)+p I'_2I_2(I'_1I_3\cdots I_N)+\cdots+pI'_NI_N(I'_1I'_2\cdots I'_{N-1})\right).
\end{equation}


Next we consider seq-tr-Hot-SVD.
For seq-tr-Hot-SVD of \(\mathcal{A}\), we perform in step \(n\) (for \(n=1,\ldots,N\)) the following
two operations.

First we compute the t-SVD of 
\[\left(\mathcal{A}*_1\widehat{\mathcal{U}}_1^H \cdots *_{n-1}\widehat{\mathcal{U}}_{n-1}^H \right)_{(n)}
\in \mathbb{C}_p^{I_n \times (I'_1 I'_2\cdots I'_{n-1}I_{n+1}I_{n+2}\cdots I_N)},\]
to get  
 \(\mathcal{U}_n \in \mathbb{C}_p^{I_n \times I_n}\). 
This step requires, assuming that \(I_n \le I'_1 I'_2\cdots I'_{n-1}I_{n+1}I_{n+2}\cdots I_N,\) 
\begin{equation}
  O\left(pI_n^2(I'_1 I'_2\cdots I'_{n-1}I_{n+1}I_{n+2}\cdots I_N)\right)
\end{equation}
flops. 
The FFT and IFFT operations are dominated. 

Next we compute \(\left(\mathcal{A}*_1\widehat{\mathcal{U}}_1^H \cdots *_{n-1}\widehat{\mathcal{U}}_{n-1}^H \right)
*_n \widehat{\mathcal{U}}_n^H\), which can be obtained by  
\[\widehat{\mathcal{U}}_n^H*
\left(\mathcal{A}*_1\widehat{\mathcal{U}}_1^H \cdots *_{n-1}\widehat{\mathcal{U}}_{n-1}^H \right)_{(n)}
=\widehat{\mathcal{U}}_n^H* \mathcal{U}_n*\Sigma_n*\mathcal{V}_n^H
=\widehat{\Sigma}_n* \widehat{\mathcal{V}}_n^H,\]
where \(\widehat{\Sigma}_n \in \mathbb{C}_p^{I'_n \times I'_n}\) and 
\(\widehat{\mathcal{V}}_n^H \in \mathbb{C}_p^{I'_n \times (I'_1 I'_2\cdots I'_{n-1}I_{n+1}I_{n+2}\cdots I_N)}\)
are obatained through truncation of the t-SVD of $\left(\mathcal{A}*_1\widehat{\mathcal{U}}_1^H \cdots *_{n-1}\widehat{\mathcal{U}}_{n-1}^H \right)_{(n)}$.
This step requires the computaion of \(I'_n  (I'_1 I'_2\cdots I'_{n-1}I_{n+1}I_{n+2}\cdots I_N)\) t-product
of tubal scalars and needs
\[O\left(p \log(p) I'_n (I'_1 I'_2\cdots I'_{n-1}I_{n+1}I_{n+2}\cdots I_N)\right)\]
flops.

In summary, a rank-\((I'_1,\ldots,I'_N)\) seq-tr-Hot-SVD of \(\mathcal{A} \in \mathbb{C}_p^{I_1 \times I_2 \times \cdots \times I_N}\) requires
\begin{equation}\label{compsthotsvd}
  \begin{aligned}
    O\Bigl(&
pI_1^2(I_2I_3\cdots I_N)+pI_2^2(I'_1I_3\cdots I_N)+\cdots+pI_N^2(I'_1I'_2\cdots I'_{N-1})\\
+&p\log(p) I'_1(I_2I_3\cdots I_N)+p \log(p)I'_2(I'_1I_3\cdots I_N)+\cdots+p\log(p)I'_N(I'_1I'_2\cdots I'_{N-1})\Bigr)
  \end{aligned}
\end{equation}
flops.

Comparing this with tr-Hot-SVD (\eqref{eq:cc:1} + \eqref{eq:cc:2}), we see that seq-tr-Hot-SVD 
requires fewer computaions than tr-Hot-SVD, especially when \(I'_1,\ldots,I'_N\) are small.

\section{Numerical Examples}\label{sec:numer}
We only conduct preliminary experiments on Hot-SVD, tr-Hot-SVD, and seq-tr-Hot-SVD in this section, as the main focus of this paper is on the derivation and providing theoretical analysis on this models. All the examples are   conducted on an Intel
i7 CPU desktop computer with 32 GB of RAM. The supporting software is Matlab 2019b. We remark that our codes are modified from those in Tensorlab \cite{vervliet2016tensorlab}, namely, the codes are tubal versions of the corresponding ones in \cite{vervliet2016tensorlab}. 
We set $L$ as DFT in this section. 

We first illustrate the properties of \cref{th:error_bound_tr_hotsvd} via a small example. Consider the tensor $\mathcal A\in\mathbb R^{2\times 2\times 2\times 2} $ given by $\mathcal A(i,j,k,l) = (i+j+k+l - 3)^{-1}$ for each $(i,j,k,l)$. This is konwn as the Hilbert tensor \cite{song2014infinite}. Applying Hot-SVD to $\mathcal A$, we obtain $\mathcal U_1=\mathcal U_2 = \mathcal U_3 = \mathcal U \in\mathbb R^{2\times 2}_2$, with $\mathcal U(:,:,1) = \left[\begin{smallmatrix} -0.8924 &   0.4395\\
   -0.4395 &  -0.8924 \end{smallmatrix}\right]$, $\mathcal U(:,:,2) = \left[\begin{smallmatrix} 0.0453 &    0.0920\\
   -0.0920 &   0.0453 \end{smallmatrix}  \right]$, and $\mathcal S\in\mathbb R^{2\times 2\times 2}_2$ with
   \begin{align*}
     &  \mathcal S(1,1,1) = [-1.4734 ~  -0.8780]^T, ~\mathcal S(2,2,2) = [0.0102~    0.0107]^T,\\
      & \mathcal S(2,1,1) = \mathcal S(1,2,1) = \mathcal S(1,1,2) = [-0.0004~ -0.0004]^T, \\
      & \mathcal S(2,2,1) = \mathcal S(2,1,2) = \mathcal S(1,2,2) = [-0.0612 ~  -0.0343]^T.
   \end{align*}
   Then, noticing the symmetry of the tubal scalars, we have
   \begin{align*}
      & \sum^2_{i_2,i_3=1}\nolimits\mathcal S(1,i_2,i_3)*\mathcal S(2,i_2,i_3) = \mathcal S(1,1,2)*(\mathcal S(1,1,1) +2\mathcal  S(1,2,2)   ) \\
       &~~~~~~~~~~~~+ \mathcal S(1,2,2)*\mathcal S(2,2,2) =[0 ~0]^T,\\
       &\sum^2_{i_1,i_3=1}\mathcal S(i_1,1,i_3)*\mathcal S(i_1,2,i_3) = \sum^2_{i_1,i_2=1}\mathcal S(i_1,i_2,1)*\mathcal S(i_1,i_2,2) = \sum^2_{i_2,i_3=1}\mathcal S(1,i_2,i_3)*\mathcal S(2,i_2,i_3) = [0~0]^T,
   \end{align*}
   illustrating the all-orthogonality of $\mathcal S$. 
   
  We then show the ordering of $\mathcal S$. Note that 
  \begin{align*}
  &\mathcal S_{(1)} = \left[\begin{smallmatrix} \mathcal S(1,1,1)& \mathcal S(1,2,1)& \mathcal S(1,1,2) & \mathcal S(1,2,2)\\
  \mathcal S(2,1,1) & \mathcal S(2,2,1)& \mathcal S(2,1,2) & \mathcal S(2,2,2)\end{smallmatrix}\right] \in\mathbb R^{2\times 4}_2,\\
  &\mathcal S_{(2)} = \left[\begin{smallmatrix} \mathcal S(1,1,1)& \mathcal S(2,1,1) & \mathcal S(1,1,2) & \mathcal S(2,1,2) \\
  \mathcal S(1,2,1) & \mathcal S(2,2,1) & \mathcal S(2,1,2) & \mathcal S(2,2,2) \end{smallmatrix} \right]  \in\mathbb R^{2\times 4}_2,\\
  & \mathcal S_{(3)} = \left[\begin{smallmatrix} \mathcal S(1,1,1) & \mathcal S(2,1,1) & \mathcal S(1,2,1) & \mathcal S(2,2,1) \\
  \mathcal S(1,1,2) & \mathcal S(2,1,2) & \mathcal S(1,2,2) & \mathcal S(2,2,2) \end{smallmatrix}   \right] \in\mathbb R^{2\times 4}_2.
  \end{align*}
  Then, 
  \begin{align*}
      (\|\mathcal S_{i_1=1}\|, \|\mathcal S_{i_1=2}\|) =  (\|\mathcal S_{i_2=1}\|  , \|\mathcal S_{i_2=2}\|) =(\|\mathcal S_{i_3=1}\|  , \|\mathcal S_{i_3=2}\|  ) =(1.7166,0.1002),
  \end{align*}
  confirming the ordering property.
  
  Next, we generate $\mathcal A$ as
  \[ \mathcal A = \mathcal A^{\sharp}/\|\mathcal A^{\sharp}\| + \beta \mathcal E/\|\mathcal E\|,   \]
  where $\mathcal A^{\sharp} = \sum^R_{i=1}\mathbf a_{1,i}\circ \cdots \circ \mathbf a_{N,i}$,    $[\mathbf a_{1,1} \cdots \mathbf a_{1,R}   ]  \in \mathbf R^{I_1\times R},\ldots, [\mathbf a_{N,1} \cdots \mathbf a_{N,R}] \in\mathbb R^{I_N\times R}$ are randomly generated matrices obeying standard Gaussian distribution, and $\mathcal E$ is also a randomly generated tensor obeying standard Gaussian distribution. We set $R=5$ and $\beta=0.1$ in our experiment. We apply tr-Hot-SVD and seq-tr-Hot-SVD on recovering $\mathcal A^{\sharp}$. The size of the core tubal tensor $\mathcal S$ is $I_1^\prime = \cdots = I_{N-1}^{\prime}=R$, and the tubal length $p=I_N$.  We evaluate $err = \|\mathcal A^{\sharp} - \bar{\mathcal A}\|/\|\mathcal A^{\sharp}\|$ and the CPU time, where $\bar{\mathcal A} = \mathcal S *_1\mathcal U_1^T \cdots *_{N-1} \mathcal U_{N-1}^T$ is generated by tr-Hot-SVD or seq-tr-Hot-SVD. The results are presented in Table \ref{tab:1}, averaged over $50$ instances for each case.
  
  From the table, we observe that both algorithms can recover the true tensors well, and seq-tr-Hot-SVD is slightly better in terms of the recovery error. Considering the CPU time, we see that seq-tr-Hot-SVD usually performs $2$ to $3$ times faster than tr-Hot-SVD, which confirms the computational complexity analysis in \cref{sec:comp}. In particular, when the size of the tensor becomes larger and larger (relative to the truncation size $I_1^\prime,\ldots, I_{N-1}^\prime$), the advantage of seq-tr-Hot-SVD turns out to be more evident. 
\begin{table}[]
    \centering
  \begin{tabular}{r|rr|rr}
  	\toprule
  	\multicolumn{1}{r}{} & \multicolumn{2}{c}{tr-Hot-SVD} & \multicolumn{2}{c}{seq-tr-Hot-SVD}\\
  	\midrule 
    {$[I_1 ~\cdots~ I_N]$} & {err}  & {time} & {err} & {time} \\
    \toprule
$[10~ 10~ 10~ 10]$    & 0.04616 & 0.010  & 0.04595  & 0.005  \\
\midrule
$[15~ 15~ 15~ 10]$    & 0.02913  & 0.015  & 0.02911 & 0.009  \\
\midrule
$[20~ 20~ 20~ 10]$    & 0.02131 & 0.022  & 0.02121 & 0.013  \\
\midrule
$[25~ 25~ 25~ 10]$    & 0.01671 & 0.036  & 0.01668  & 0.020  \\
\midrule
$[30~ 30~ 30~  10]$    & 0.01371  & 0.041  & 0.01370 & 0.030  \\
\midrule
$[35~ 35~ 35~  10]$   & 0.01165  & 0.048  & 0.01161 & 0.036  \\
\midrule
$[40~ 40~ 40~ 10]$    & 0.01023 & 0.054  & 0.01013  & 0.039  \\
\midrule
$[10~ 10~ 10~ 10~  10]$    & 0.02725  & 0.023  & 0.02721  & 0.014  \\
\midrule
$[15~ 15~ 15~ 15~   10]$    & 0.01291 & 0.042  & 0.01284 & 0.040  \\
\midrule
$[20~ 20~ 20~ 20~ 10]$    & 0.00773 & 0.096 & 0.00768  & 0.066  \\
\midrule
$[25~25~25~25~10]$& 0.00518 & 	0.212 & 	0.00516 & 	0.112\\ 
\midrule
$[30~30~30~30~10]$& 0.00374 &	0.441 &	 0.00373 & 	0.205 \\ 
\midrule
$[35~35~35~35~10]$& 0.00289 & 	0.771 &	 0.00287 & 	0.334 \\ 
\midrule
$[40~40~40~40~10]$& 0.00230 & 	1.328 & 	0.00229 & 	0.548 \\
\midrule
$[10~ 10~ 10~ 10~ 10~ 10]$    & 0.01812  & 0.084  & 0.01809  & 0.066  \\
\midrule
$[15~ 15~ 15~ 15~ 15~ 10]$    & 0.00669 & 0.542  & 0.00667  & 0.277  \\
\midrule
$[20~ 20~ 20~ 20~ 20~ 10]$    & 0.00332  & 2.203  & 0.00331  & 0.880 \\
\midrule
$[25~ 25~ 25~ 25~ 25~ 10]$    & 0.00194 & 7.134  & 0.00193  & 2.585  \\
\midrule
$[30~ 30~ 30~ 30~ 30~ 10]$    & 0.00126 & 15.889  & 0.00125 & 5.629 \\
\bottomrule 
\end{tabular}%
\caption{Comparisons of tr-Hot-SVD and seq-tr-Hot-SVD on recovering randomly generated tensors. $I_1^\prime = \cdots = I_{N-1}^\prime = 10$.}\label{tab:1}
\end{table}

We then test tr-Hot-SVD and seq-tr-Hot-SVD on color video compression. 
The tested video ``airport'' was downloaded from \url{http://perception.
i2r.a-star.edu.sg/bk_model/bk_index.html}. The original video consists of $4583$
frames, each of size $144 \times 176$. We use 500 frames, resulting into a tensor of size $500\times 144\times 176\times 3$. We treat it as a tubal tensor in $\mathbb R^{500\times 144\times 176}_3$. We set different truncation sizes for the two algorithms to compress the data. We use tr-HOSVD \cite{de2000a} and seq-tr-HOSVD \cite{vannieuwenhoven2012new} as  baselines. The truncation sizes of tr-HOSVD and seq-tr-HOSVD are the same as the tubal counterparts, except that we do not truncate the fourth-mode of both tr-HOSVD and seq-tr-HOSVD. We   evaluate the reconstruction error $err = \| \mathcal A^{\sharp} - \bar{\mathcal A}\|/\|\mathcal A^{\sharp}\|$ and the CPU time, where $\mathcal A^{\sharp}$ is the data tensor and $\bar{\mathcal A}$ is reconstructed by the algorithms. The results are presented in Table \ref{tab:2}. Some reconstructed frames by seq-tr-Hot-SVD are illustrated in Fig. \ref{fig:color_video}.

\begin{table}[]
    \centering
  \begin{tabular}{r|rr|rr|rr|rr}
  	\toprule
  	\multicolumn{1}{r}{} & \multicolumn{2}{c}{tr-Hot-SVD} & \multicolumn{2}{c}{seq-tr-Hot-SVD} & \multicolumn{2}{c}{tr-HOSVD} & \multicolumn{2}{c}{seq-tr-HOSVD}\\
  	\midrule 
    {$[I^\prime_1 ~I^\prime_2~ I^\prime_3]$} & {err}  & {time} & {err} & {time} & err & time & err & time \\
    \toprule
$[200~ 50~ 50]$ & \underline{0.0769}  & 14.50  & {\bf 0.0762}  & 7.37  & 0.0782  & 5.33  & 0.0776  & 3.20  \\
$[100~ 50~ 50]$ & \underline{0.0841}  & 13.85  & {\bf 0.0835}  & 6.35  & 0.0854  & 5.59  & 0.0848  & 2.66  \\
$[50~ 50~ 50]$ & \underline{0.0943}  & 13.19  & {\bf 0.0937}  & 9.98  & 0.0955  & 6.07  & 0.0950  & 2.95  \\
$[30~ 30~ 30]$ & \underline{0.1171}  & 17.75  & {\bf 0.1157}  & 7.94  & 0.1187  & 5.68  & 0.1174  & 2.19  \\
$[20~ 10~ 10]$ & \underline{0.1595}  & 12.48  & {\bf 0.1574}  & 3.39  & 0.1621  & 6.01  & 0.1601  & 1.76  \\
$[10~ 5~ 5]$ & \underline{0.1841}  & 13.65  & {\bf 0.1824}  & 3.27  & 0.1881  & 6.03  & 0.1863  & 1.46  \\
\bottomrule
\end{tabular}%
\caption{Comparisons of tr-Hot-SVD, seq-tr-Hot-SVD, tr-HOSVD, and seq-tr-HOSVD on color video compression.} \label{tab:2}
\end{table}

From the table, we observe that concerning the reconstruction error, seq-tr-Hot-SVD is the best among the four algorithms, followed by tr-Hot-SVD. This shows the advantage of the tubal versions of HOSVD. Concerning the CPU time, seq-tr-HOSVD is the fastest one, while seq-tr-Hot-SVD and tr-Hot-SVD are slower than tr-HOSVD and seq-tr-HOSVD. This is because the tubal versions need to perform additional Fourier transforms. We also see that seq-tr-Hot-SVD is about $2$ to $4$ times faster than tr-Hot-SVD.

	\begin{figure} 
		\centering

		\renewcommand*{\arraystretch}{0.5}
		\setlength{\tabcolsep}{1pt}

		\begin{tabular}{ccccccc}

			\includegraphics[width=0.7in]
			{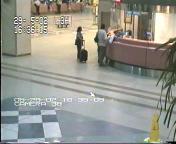}
			
			&

			\includegraphics[width=0.7in]
			{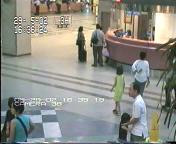}
			
			&
			\includegraphics[width=0.7in]
			{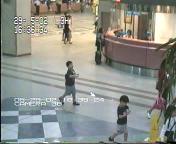}
			
			&
			\includegraphics[width=0.7in]
			{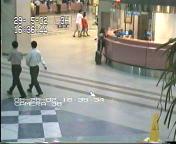}
			
			&
			\includegraphics[width=0.7in]
			{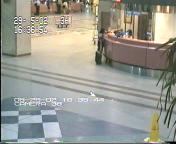}
			
			&
			\includegraphics[width=0.7in]
			{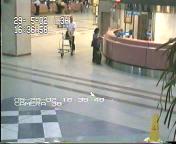}
			
			&
			\includegraphics[width=0.7in]
			{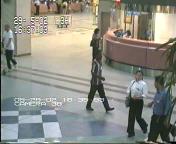}
			\\

			\includegraphics[width=0.7in]
			{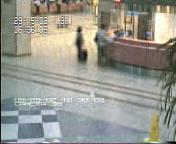}
			
			&
			
			\includegraphics[width=0.7in]
			{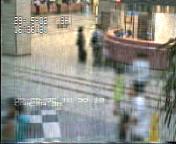}
			
			&
			\includegraphics[width=0.7in]
			{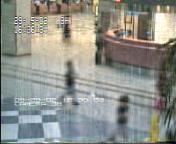}
			
			&
			\includegraphics[width=0.7in]
			{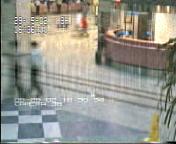}
			
			&
			\includegraphics[width=0.7in]
			{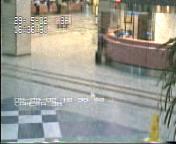}
			
			&
			\includegraphics[width=0.7in]
			{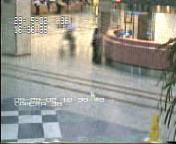}
			
			&
			\includegraphics[width=0.7in]
			{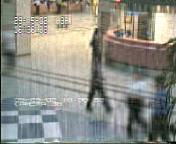}
			\\

			\includegraphics[width=0.7in]
			{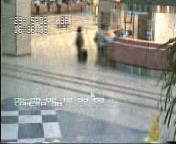}
			
			&

			\includegraphics[width=0.7in]
			{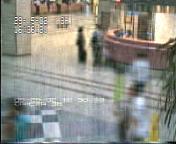}
			
			&
			\includegraphics[width=0.7in]
			{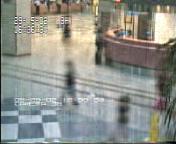}
			
			&
			\includegraphics[width=0.7in]
			{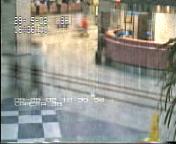}
			
			&
			\includegraphics[width=0.7in]
			{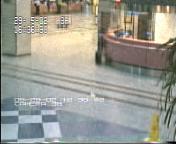}
			
			&
			\includegraphics[width=0.7in]
			{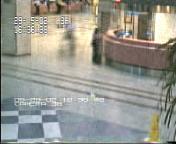}
			
			&
			\includegraphics[width=0.7in]
			{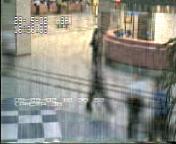}
			\\
			
			\includegraphics[width=0.7in]
			{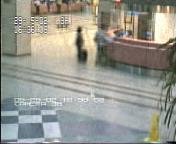}
			
			&

			\includegraphics[width=0.7in]
			{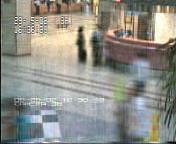}
			
			&
			\includegraphics[width=0.7in]
			{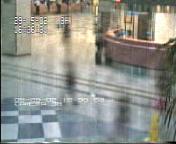}
			
			&
			\includegraphics[width=0.7in]
			{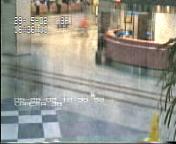}
			
			&
			\includegraphics[width=0.7in]
			{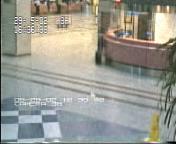}
			
			&
			\includegraphics[width=0.7in]
			{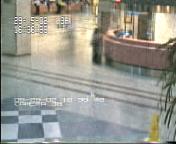}
			
			&
			\includegraphics[width=0.7in]
			{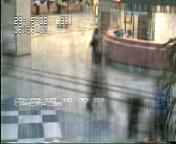}
			\\

			\includegraphics[width=0.7in]
			{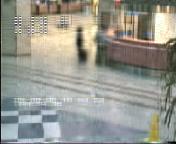}
			
			&

			\includegraphics[width=0.7in]
			{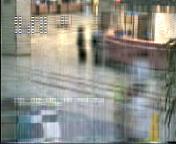}
			
			&
			\includegraphics[width=0.7in]
			{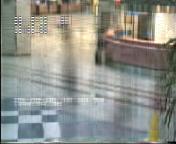}
			
			&
			\includegraphics[width=0.7in]
			{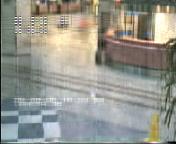}
			
			&
			\includegraphics[width=0.7in]
			{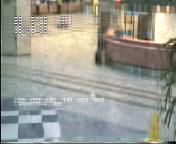}
			
			&
			\includegraphics[width=0.7in]
			{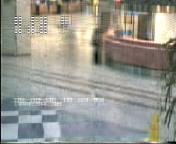}
			
			&
			\includegraphics[width=0.7in]
			{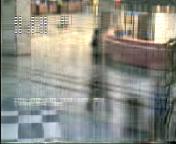}
			\\

			\includegraphics[width=0.7in]
			{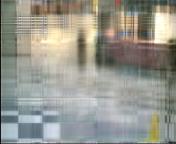}
			
			&

			\includegraphics[width=0.7in]
			{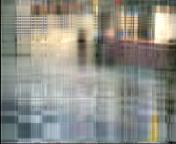}
			
			&
			\includegraphics[width=0.7in]
			{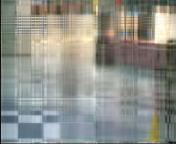}
			
			&
			\includegraphics[width=0.7in]
			{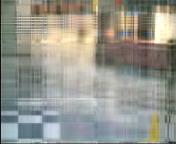}
			
			&
			\includegraphics[width=0.7in]
			{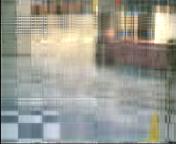}
			
			&
			\includegraphics[width=0.7in]
			{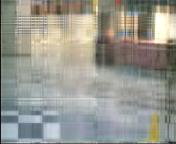}
			
			&
			\includegraphics[width=0.7in]
			{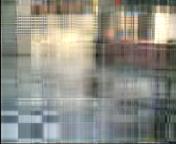}
			\\
			
			\includegraphics[width=0.7in]
			{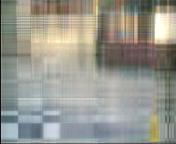}
			
			&

			\includegraphics[width=0.7in]
			{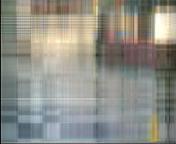}
			
			&
			\includegraphics[width=0.7in]
			{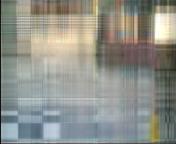}
			
			&
			\includegraphics[width=0.7in]
			{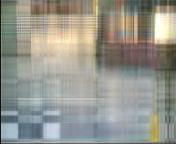}
			
			&
			\includegraphics[width=0.7in]
			{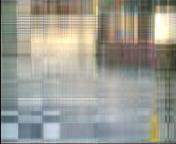}
			
			&
			\includegraphics[width=0.7in]
			{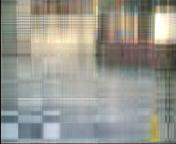}
			
			&
			\includegraphics[width=0.7in]
			{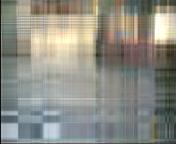}
		\end{tabular}

		\caption{The first row: the original frames. The second to the last rows: frames reconstructed by seq-tr-Hot-SVD with different truncation sizes.}\label{fig:color_video}
	\end{figure}

\section{Conclusions}
\label{sec:con}

In this paper we studied the analogue of HOSVD in the setup of tubal tensors, which we call Hot-SVD. This is alternative to the t-SVD for   tensors of order higher than three  \cite{martin2013order}.
To   prove   the validness of Hot-SVD, we introduced a new transpose for third-order tensors and established some basic properties.
Hot-SVD enjoys most of the properties of the HOSVD and manifests the efficacy of the language of tubal matrices.  
We also established an error bound $\sqrt{N}$ for the truncated   Hot-SVD and sequentially truncated Hot-SVD of tensors order $N+1$. We remark that our purpose of studying Hot-SVD is not to compare it with t-SVD for tensors of order higher than three \cite{martin2013order}, but to contribute another kind of decomposition of higher-order tensors in the tensor-tensor product setting.  
  In the future, randomized algorithms can be investigated for Hot-SVD.

\section*{Acknowledgments}
This work was supported by   National Natural Science
Foundation of China Grant 12171105,     Fok Ying Tong Education Foundation Grant 171094, and the special foundation for Guangxi Ba Gui Scholars.  
\bibliographystyle{plain}
\bibliography{references,t_prod,tensor,TensorCompletion,orth_tensor}

\end{document}